\documentclass[11pt]{amsart}
\usepackage[latin1]{inputenc}
\usepackage[T1]{fontenc}
\usepackage{amsfonts}
\usepackage{amsmath}
\usepackage{amssymb}
\usepackage{graphicx}
\usepackage{pstricks}
\usepackage{pst-grad}
\usepackage{multido}
\usepackage{color}
\usepackage{amsthm}
\usepackage{geometry}
\usepackage[absolute]{textpos}
\definecolor{purple}{rgb}{0.8,0.12,0.8}
\definecolor{orange}{rgb}{1.0,0.7,0.0}
\definecolor{pink}{rgb}{1,0.5,0.8}
\definecolor{blackg}{rgb}{0.1,0.25,0.1}
\definecolor{ForestGreen}{cmyk}{0.91,0,0.88,0.42}
\definecolor{Turquoise}{cmyk}{0.85,0,0.20,0}

\newcommand{\cA}{\mathcal{A}}
\newcommand{\cB}{\mathcal{B}}
\newcommand{\cC}{\mathcal{C}}

\newcommand{\cF}{\mathcal{F}}

\newcommand{\cH}{\mathcal{H}}

\newcommand{\cJ}{\mathcal{J}}

\newcommand{\cL}{\mathcal{L}}
\newcommand{\cLR}{\mathcal{LR}}
\newcommand{\cM}{\mathcal{M}}

\newcommand{\cP}{\mathcal{P}}

\newcommand{\cR}{\mathcal{R}}

\newcommand{\cU}{\mathcal{U}}

\newcommand{\sg}{\mathcal{h}}
\newcommand{\sd}{\mathcal{i}}

\newcommand{\bS}{\mathbf{S}}

\newcommand{\ba}{\mathbf{a}}

\newcommand{\fB}{\mathfrak{B}}
\newcommand{\fC}{\mathfrak{C}}

\newcommand{\fH}{\mathfrak{H}}

\newcommand{\fS}{\mathfrak{S}}

\newcommand{\nZ}{\mathbb{Z}}

\newcommand{\nN}{\mathbb{N}}

\newcommand{\la}{\lambda}

\newcommand{\eps}{\epsilon}

\newcommand{\tA}{\tilde{A}}
\newcommand{\tB}{\tilde{B}}
\newcommand{\tC}{\tilde{C}}

\newcommand{\tF}{\tilde{F}}
\newcommand{\tG}{\tilde{G}}
\newcommand{\tW}{\tilde{W}}
\newcommand{\tJ}{\tilde{J}}
\newcommand{\tL}{\tilde{L}}

\newcommand{\tc}{\tilde{c}}
\newcommand{\tb}{\tilde{b}}
\newcommand{\tT}{\tilde{t}}

\newcommand{\sqs}{\sqsubset}
\newcommand{\sq}{\sqsubseteq}
\newcommand{\ov}{\overline}
\newcommand{\ind}{\underset}

\newtheorem{Th}{Theorem}[section]

\newtheorem{Lem}[Th]{Lemma}

\newtheorem{Prop}[Th]{Proposition}
\newtheorem{Def-Prop}[Th]{Definition-Proposition}
\newtheorem{conj}[Th]{Conjecture}
\theoremstyle{definition}

\newtheorem{Cl}[Th]{Claim}
\theoremstyle{remark}
\newtheorem{Rem}[Th]{Remark}

\begin{document}

\title{Kazhdan-Lusztig cells in affine Weyl groups of rank $2$}
\author{J\'er\'emie Guilhot}
\address{School of Mathematics and Statistics F07
University of Sydney NSW 2006
Australia }
\email{guilhot@maths.usyd.edu.au}

\date{January, 2009}
\begin{abstract}
In this paper we determine the partition into Kazhdan-Lusztig cells of the affine Weyl groups of type $\tB_{2}$ and $\tG_{2}$ for any choice of parameters. Using these partitions we show that the semicontinuity conjecture of Bonnaf\'e holds for these groups.
\end{abstract}

\maketitle

\section{Introduction}
Let $W$ be a Coxeter group with generating set $S$. Following Lusztig, let $L$ be a weight function on $W$, that is a function $L:W\longrightarrow \nZ$ such that $L(ww')=L(w)+L(w')$ whenever $\ell(ww')=\ell(w)+\ell(w')$ (where $\ell$ is the usual length function in $(W,S)$). This gives rise to various pre-order relations $\leq_{\cL}$, $\leq_{\cR}$ and $\leq_{\cLR}$ which in turn give rise to Kazhdan-Lusztig left, right and two-sided cells, respectively. Cells for an arbitrary Coxeter group were first defined in the equal parameter case (i.e. $L=\ell$) by Kazhdan and Lusztig \cite{KL1} and subsequently in the unequal parameter case by Lusztig \cite{Lus1p}. Their construction relies on the existence of the Kazhdan-Lusztig basis in the corresponding Iwahori-Hecke algebra. Cells are known to play a fundamental role in the representation theory of reductive groups over finite or $p$-adic fields.

In this paper, we are especially interested in the partition of affine Weyl groups into cells. In the equal parameter case the situation is relatively well understood due to the fact that there is a geometric interpretation of the Kazhdan-Lusztig basis. This interpretation yields some positivity properties such as the positivity of the coefficients of the structure constants with respect to the Kazhdan-Lusztig basis. The cells have been explicitly described for type $\tilde{A}_{r},r\in\mathbb{N}$  \cite{Lus2, Shi1}, ranks 2, 3 \cite{Bed,Lus3,Du} and types $\tilde{B}_{4}$ \cite{B4}, $\tilde{C}_{4}$ \cite{Shi4}, $\tilde{D}_{4}$ \cite{Chen,Shi3} and $\tF_{4}$ \cite{Shi2,Shi5}.

In the unequal parameter case a general geometric interpretation does not yet exist and the positivity properties do not hold anymore. Therefore, the knowledge in this case is nowhere near the one in the equal parameter case. The cells have been explicitly described in type $\tA_{1}$ for all parameters \cite{bible}, in type $\tB_{2}$ for some specific choices of parameters which still admit a geometric interpretation \cite{bremke} and in type $\tG_{2}$ in the so-called asymptotic case \cite{jeju3}.


The aims of this paper are as follows. First, given an affine Weyl group $W$ and a weight function $L$ on $W$, we will present an algorithm which determines a partition of $W$ into finitely many pieces. We conjecture that this partition is precisely the partition of $W$ into cells with respect to $L$. We will be able to deduce this conjecture  from some general conjectures of Lusztig \cite{conn,bible}. Our conjecture (if true in general) reduces the problem of determining the cells of $W$ to the analogous problem for the finite parabolic  subgroups of $W$. We then apply this algorithm to affine Weyl groups of type $\tB_{2}$ and $\tG_{2}$ and show that, for any weight function, the resulting partition is indeed the partition into cells. The proofs rely on some general methods we developed in \cite{jeju1,jeju2,jeju3}, and some explicit computer calculations using GAP \cite{GAP}. These results for $W$ of type $\tB_{2}$ and $\tG_{2}$ provide the first substantial examples for Bonnaf\'e's semicontinuity conjecture \cite{semi} in the affine case.


\section{Generalities}
In this section $(W,S)$ denotes an arbitrary Coxeter system (with $|S|<\infty$) together with a weight function $L$. 


\subsection{Hecke algebras and Kazhdan-Lusztig basis}
Let $\cA=\nZ[v,v^{-1}]$ where $v$ is an indeterminate. Let $\cH$ be the Iwahori-Hecke algebra associated to $W$, with $\cA$-basis  $\{T_{w}|w\in W\}$ and multiplication rule given by
\begin{equation*}
T_{s}T_{w}=
\begin{cases}
T_{sw}, & \mbox{if } \ell(sw)>\ell(w),\\
T_{sw}+(v^{L(s)}-v^{-L(s)})T_{w}, &\mbox{if } \ell(sw)<\ell(w),
\end{cases}
\end{equation*}
for all $s\in S$ and $w\in W$. Let $\bar\ $ be the involution of $\cA$ which takes $v$ to $v^{-1}$. It can be extended to a ring involution of $\cH$ via
$$\ov{\sum_{w\in W} a_{w}T_{w}}=\sum_{w\in W} \bar{a}_{w}T_{w^{-1}}^{-1}\quad (a_{w}\in\cA).$$
We set $\cA_{<0}=v^{-1}\nZ[v^{-1}]$ and  $\cA_{\leq 0}=\nZ[v^{-1}]$. Let $w\in W$. There exists a unique element $C_{w}\in\cH$ (see \cite[Theorem 5.2]{bible}) such that
\begin{enumerate}
\item $\bar{C}_{w}=C_{w}$
\item $C_{w}\in T_{w}+\bigoplus_{y\in W} \cA_{<0} T_{y}$
\end{enumerate}
For any $w\in W$ we set 
$$C_{w}=T_{w}+\sum_{y\in W } P_{y,w} T_{y}\quad \text{where $P_{y,w}\in \cA_{< 0}$}.$$
It is well known (\cite[\S 5.3]{bible}) that $P_{y,w}=0$ whenever $y\nleq w$  (here $\leq$ denotes the Bruhat order). It follows that $\{C_{w}|w\in W\}$ forms a basis of $\cH$ (the ``Kazhdan-Lusztig basis''). The coefficients $P_{y,w}$ are known as the Kazhdan-Lusztig polynomials. 
\begin{Rem}
Note that the element $C_{w}$ is denoted by $C'_{w}$ in \cite{KL1,Lus1p} and by $c_{w}$ in \cite{bible}. The Kazhdan-Lusztig polynomials are denoted by $P^{*}_{y,w}$ in \cite{Lus1p} and by $p_{y,w}$ in \cite{bible}. Originally, the Kazhdan-Lusztig polynomials were defined by $v^{\ell(w)-\ell(y)}P_{y,w}$ in  \cite[1.1.c]{KL1}.
\end{Rem}


\subsection{Kazhdan-Lusztig cells}
Since the elements $C_{w}$ ($w\in W$) form a basis of $\cH$ we can write, for all $x,y\in W$
$$C_{x}C_{y}=\sum_{z\in W} h_{x,y,z}C_{z}$$
where $h_{x,y,z}\in \cA$ and $\bar{h}_{x,y,z}=h_{x,y,z}$. \\
We write $z\leftarrow_{\cL} y$ if there exists $s\in S$ such that $h_{s,y,z}\neq 0$, that is $C_{z}$ appears with a non-zero coefficient in the expression of $C_{s}C_{y}$ in the Kazhdan-Lusztig basis. The Kazhdan-Lusztig left pre-order $\leq_{\cL}$ on $W$ is the transitive closure of this relation. 
The equivalence relation associated to $\leq_{\cL}$ will be denoted by $\sim_{\cL}$, that is
$$x\sim_{\cL}y \Longleftrightarrow x\leq_{\cL}y\text{ and } y\leq_{\cL} x\quad (x,y\in W).$$
The corresponding equivalence classes are called the left cells of $W$. Similarly, we define $\leq_{\cR}$, $\sim_{\cR}$ and right cells, multiplying on the right. We have (see \cite[\S 8]{bible})
$$x\leq_{\cL} y \Longleftrightarrow x^{-1}\leq_{\cR} y^{-1}.$$
\begin{Rem}
Let $x\in W$ and $s\in S$ be such that $sx>x$. Then, following \cite[Theorem 6.6]{bible}, we have $sx\leq_{\cL} x$. Similarly if $xs>x$ we have $xs\leq_{\cR} x$.
\end{Rem}
Finally we write $x\leq_{\cLR} y$ if there exists a sequence $x=x_{0},x_{1},...,x_{n}=y$ of $W$ such that for each $0\leq i\leq n-1$ we have either $x_{i}\leftarrow_{\cL} x_{i+1}$ or $x_{i}\leftarrow_{\cR} x_{i+1}$. The equivalence relation associated to $\leq_{\cLR}$ will be denoted by $\sim_{\cLR}$ and the equivalence classes are called the two-sided cells of $W$. \\
The pre-orders $\leq_{\cL},\leq_{\cR},\leq_{\cLR}$ induce a partial order on the left, right and two-sided cells, respectively.


\subsection{Parabolic subgroups}
Let $I\subset S$. We denote by $W_{I}$ the standard parabolic subgroup generated by $I$. If $W_{I}$ is finite, we denote by $w_{I}$ the longest element of $W_{I}$. We set
$$X_{I}:=\{x\in W| \ell(xs)=\ell(x)+1 \text{ for all $s\in I$}\}.$$
Then, by \cite[Proposition 2.1.1]{geckpfeiffer}, there exists a bijection
$$\begin{array}{cccc}
X_{I}\times W_{I}&\longrightarrow &W\\
(x,u)&\longmapsto &xu
\end{array}$$
such that $\ell(xu)=\ell(x)+\ell(u)$.


\subsection{Lusztig's $\ba$-function}
In the remainder of this section, {\it we assume that $L(s)>0$ for all $s\in S$}.
We say that $W$ is bounded if there exists $N\in \nN$ such that 
$$v^{-N}h_{x,y,z}\in \cA_{\leq 0}\quad\text{for all $x,y,z\in W$.}$$
Let $N=\max_{I} L(w_{I})$ where $I$ runs over all the subsets of $S$ such that $W_{I}$ is finite. Then Lusztig has conjectured that $N$ should be a bound for $W$. It is well known that $N$ is a bound for $W$ if $W$ is finite or if $W$ is an affine Weyl group \cite{bremke,bible}.  \\
From now on, we assume that $W$ is bounded, so that the next definition is valid. Let $z\in W$. Lusztig defined the following function
$$\ba(z)=\min\{n\in\nN\ |\ v^{-n}h_{x,y,z}\in \cA_{\leq 0}\text{ for all $x,y\in W$}\}.$$
We have (see \cite[Proposition 13.8]{bible})
\begin{Prop}
\label{afonc}
Assume that $W$ is finite and let $w_{0}$ be the longest element of $W$. We have
\begin{enumerate}
\item[(a)] $\ba(w_{0})=L(w_{0})$;
\item[(b)] $\ba(w)<L(w_{0})$ for all $w<w_{0}$.
\end{enumerate}
\end{Prop}
\noindent
Let $I\subset S$. We write $\ba_{I}:W_{I}\rightarrow \nN$ for the $\ba$-function defined in terms of $W_{I}$.


\subsection{Lusztig's conjectures}
\label{Lusconj}
Lusztig has formulated 15 conjectures (\cite[\S 14]{bible}), including the following
\begin{itemize}
\item[\bf P4.] If $z'\leq_{\cLR} z$ then $\ba(z')\geq \ba(z)$. Hence, if
$z'\sim_{\cLR} z$, then $\ba(z)=\ba(z')$.
\item[\bf P9.] If $z'\leq_{\cL} z$ and $\ba(z')=\ba(z)$ then $z'\sim_{\cL}z$.
\item[\bf P10.] If $z'\leq_{\cR} z$ and $\ba(z')=\ba(z)$ then $z'\sim_{\cR}z$.
\item[\bf P11.] If $z'\leq_{\cLR} z$ and $\ba(z')=\ba(z)$ then
$z'\sim_{\cLR}z$.
\item[\bf P12.] Let $I\subset S$. If $y\in W_{I}$ then we have $\ba_{I}(y)=\ba(y)$.
\item[\bf P14.] For any $z\in W$ we have $z \sim_{\cLR} z^{-1}$.
\end{itemize}
These conjectures are known to be true in the equal parameter case when $W$ is a Weyl group (finite or affine). The proof relies on the geometric interpretation of the Kazhdan-Lusztig basis \cite{Springer}.
\begin{Rem}
\label{propp}
(1) Using the fact that the map $x\mapsto x^{-1}$ sends left cells to right cells, one can see that  {\bf P9} is equivalent to  {\bf P10}.\\
(2) Conjectures {\bf P4}, {\bf P9} and {\bf P10} imply {\bf P11} (see \cite[\S 14.11]{bible}).\\
(3) If {\bf P4} and {\bf P9} hold then for all $x,y\in W$ we have 
$$({\bf P})\hspace{1cm} x\leq_{L} y \text{ and } x\sim_{LR} y \Longrightarrow x\sim_{L}y.$$
(4) Similarly, if {\bf P4} and {\bf P10} hold then we have for all $x,y\in W$
$$({\bf P'})\hspace{1cm} x\leq_{R} y \text{ and } x\sim_{LR} y \Longrightarrow x\sim_{R}y.$$
\end{Rem}



\subsection{The $\ba'$-function}
\label{aprime}
Let $w\in W$. Let $Z(w)$ to be the set of all $u\in W$ such that
\begin{enumerate}
\item there exist $x,y\in W$ such that $w=xuy$;
\item there exists $I\subset S$ such that $W_{I}$ is finite and $u\in W_{I}$;
\item $\ell(w)=\ell(x)+\ell(u)+\ell(y)$.
\end{enumerate}
Set $\ba'(w)=\max_{u\in W} \ba(u)$. Then Lusztig conjectured \cite[\S 13.12]{bible}
$$\begin{array}{cl}
({\bf C1})& \ba(w)=\ba'(w)  \\
\end{array}$$
for all $w\in W$. As a consequence, if one assumes that {\bf P11} holds, we obtain that any two-sided cell should meet some finite standard parabolic subgroup and hence, the number of two-sided cells should be finite.


\subsection{Connected sets and cells}
Let $K$ be a subset of $W$ and let $x,y\in K$. We say that $x,y$ are left-connected (respectively right-connected) in $K$ if there exists a sequence $x=x_{0},...,x_{n}=y$ in $K$ such that $x_{i}x_{i+1}^{-1}\in S$ (respectively $x_{i}^{-1}x_{i+1}\in S$) for all $0\leq i\leq n-1$. We say that $x,y$ are connected in $K$ if  there exists a sequence $x=x_{0},...,x_{n}=y$ in $K$ such that either $x_{i}x_{i+1}^{-1}\in S$ or $x_{i}^{-1}x_{i+1}\in S$ for all $0\leq i\leq n-1$. These are equivalence relations on $K$ hence we obtain a partition of $K$ into left-connected components, right-connected components and connected components. \\
A subset $K$ of $W$ is said to be left-connected if all $x,y\in K$ are left-connected in $K$. Similarly we define right-connected subsets of $W$ and connected subsets of $W$.

\begin{Lem}
\label{connectedcell}
Assume that {\bf P9}--{\bf P11} hold. Then we have
\begin{enumerate}
\item Let $K$ be a left-connected subset of $W$ such that the $\ba$-function is constant on $K$. Then $K$ is included in a left cell.
\item Let $K$ be a right-connected subset of $W$ such that the $\ba$-function is constant on $K$. Then $K$ is included in a right cell.
\item Let $K$ be a connected subset of $W$ such that the $\ba$-function is constant on $K$. Then $K$ is included in a two-sided cell.
\end{enumerate}
\end{Lem}
\begin{proof}
We prove (1). Let $x,y\in K$ and $x=x_{0},...,x_{n}=y$ be such that 
 $x_{i}x_{i+1}^{-1}\in S$ for all $0\leq i\leq n-1$. Let $0\leq i\leq n-1$ and set $s=x_{i}x_{i+1}
^{-1}\in S$. There is two cases to consider:
\begin{itemize}
\item $sx_{i}=x_{i+1}$ and $\ell(sx_{i})=\ell(x_{i})+1$;
\item $sx_{i+1}=x_{i}$ and $\ell(sx_{i+1})=\ell(x_{i+1})+1$.
\end{itemize} 
In the first case we have $x_{i+1}\leq_{\cL} x_{i}$ and since $\ba(x_{i})=\ba(x_{i+1})$ we obtain that $x_{i}\sim_{\cL} x_{i+1}$ (using {\bf P9}). The second case is similar. The proof of (2) and (3) are similar using {\bf P10} and {\bf P11}, respectively.
\end{proof}

In the case where $W$ is an irreducible Weyl group (finite or affine), Lusztig conjectured (\cite{conn})
$$\begin{array}{cl}
({\bf C2})& \text{The left cells of $W$ are left-connected.} 
\end{array}$$
Using the fact that the map $x\mapsto x^{-1}$ sends left cells to right cells, one can easily see that if the left cells are left-connected then the right cells are right-connected. 

\begin{Lem}
\label{lcells}
Assume that {\bf P} and {\bf C2} hold. Then the two-sided cells of $W$ are connected. Furthermore, the left cells lying in a given two-sided cell are the left-connected components.
\end{Lem}
\begin{proof}
Let $C$ be a two-sided cell and let $x,y\in C$. There exists a sequence $x=x_{0},...,x_{n}=y$ in $C$ such that $x_{i}\leftarrow_{\cL} x_{i+1}$ or $x_{i}\leftarrow_{\cR} x_{i+1}$ for all $i$. Using {\bf P} (or its right version {\bf P'}) we get $x_{i}\sim_{\cL} x_{i+1}$ or $x_{i}\sim_{\cR} x_{i+1}$ for all $i$. It follows by {\bf C2} that two-sided cells are connected. \\
Next let $K$ be a left-connected component of $C$. Since $C$ is a union of left cells and left cells are left-connected it follows that $K$ has to be a union of left cells. Let $x,y\in K$. There exists a sequence $x=x_{0},...,x_{n}=y$ in $K$ such that 
 $x_{i}x_{i+1}^{-1}\in S$ for all $0\leq i\leq n-1$. Let $0\leq i\leq n-1$ and set $s=x_{i}x_{i+1}
^{-1}\in S$. Then we have either $sx_{i}=x_{i+1}$ and $\ell(sx_{i})=\ell(x_{i})+1$ or $sx_{i+1}=x_{i}$ and $\ell(sx_{i+1})=\ell(x_{i+1})+1$.
It follows that $x_{i}\leq_{\cL} x_{i+1}$ or $x_{i+1}\leq_{\cL} x_{i}$. Using {\bf P} we get $x_{i}\sim_{\cL} x_{i+1}$ and $K$ is included in a left cell. It follows that $K$ is a left cell.
\end{proof}


\subsection{Example: the dihedral groups}
\label{dihedral}
In this section, let $W$ be a Coxeter group of type $I_{2}(m)$ ($m\geq 2$ possibly $m=\infty$) with generators $s_{1},s_{2}$ such that $(s_{1}s_{2})^m=1$ (if $m<\infty$). If $m=\infty$ then $I_{2}(\infty)$ is the infinite dihedral group, which is also the affine Weyl group of type $\tA_{1}$. Let $L$ be a weight function on $W$. We set $L(s_{1})=a\in\nN$ and $L(s_{2})=b\in\nN$. We may assume without loss of generality that $a\geq b$. We have
\begin{Th}[du Cloux \cite{duC}, Geck \cite{rem}, Lusztig \cite{bible}]
Conjectures {\bf P1}--{\bf P15} hold for $W$ and $L$.
\end{Th}
For any $k\geq 0$ we set $1_{k}=s_{1}s_{2}...$ ($k$ factors) and $2_{k}=s_{2}s_{1}...$ ($k$ factors). Note that, if $m<\infty$ we have $1_{m}=2_{m}$. 
In the following table, we describe the partition of $W$ into two-sided cells, the left cells lying in a given two-sided cell and the values of the $\ba$-function. From this table, one can easily check that that {\bf C1} and {\bf C2} hold for dihedral groups.

\begin{table}[h!] \caption{Cells in the dihedral groups} 
{\small
\begin{center}
$
\renewcommand{\arraystretch}{1.4}
\begin{array}{|c|c|c|c|} \hline
W & \text{two-sided cells} & \text{ left cells }& \text{$\ba$-function}\\ \hline
 & \{1_{0}\} &  \{1_{0}\} & 0 \\
I_{2}(m), a=b & W-\{1_{0},1_{m}\}& \{1_{1},2_{2},...\},\{2_{1},1_{2},...\}& a \\
 & \{1_{m}\}& \{1_{m}\}& ma \\ \hline
 & \{1_{0}\} &  \{1_{0}\} & 0 \\
I_{2}(m), a>b & \{2_{1}\}& \{2_{1}\}& b \\
 m\geq 3, \text{$m$ even}& W-\{1_{0},2_{1},1_{m-1},1_{m}\}&\{1_{2},...,2_{m-1}\},\{1_{1},...,2_{m-2}\}& a \\
 & \{1_{m-1}\} & \{1_{m-1}\} & \frac{m}{2}a-\frac{m}{2}b+b\\
 & \{1_{m} \}&  \{1_{m}\} & \frac{m}{2}a+\frac{m}{2}b\\ \hline
& \{1_{0}\} &  \{1_{0}\} & 0 \\
I_{2}(2) & \{2_{1}\}& \{2_{1}\}& b \\
a\geq b& \{1_{1}\} &\{1_{1}\}& a \\
& \{1_{2}\} &  \{1_{2}\} & a+b \\ \hline
\tA_{1}& \{1_{0}\} &  \{1_{0}\} & 0 \\
 a=b & W-\{1_{0}\}& \{1_{1},2_{2},...\},\{2_{1},1_{2},...\}& a \\ \hline
 & \{1_{0}\} &  \{1_{0}\} & 0 \\
\tA_{1}, a>b  & \{2_{1}\}& \{2_{1}\}& b \\ 
& W-\{1_{0},2_{1}\}& \{1_{1},2_{2},...\},\{1_{2},2_{3}...\}& a \\ \hline
\end{array}
$
\end{center}}
\end{table}


\section{On the decomposition of an affine Weyl group into cells}
Let $W$ be an irreducible affine Weyl group together with a positive weight function $L$. In this section we present an algorithm for determining the partition of $W$ into cells from the knowledge of cells in all proper parabolic subgroups. \\
Throughout this section, we assume that {\bf P4}, {\bf P12} and {\bf C2} hold for all proper parabolic subgroups of $W$.

\subsection{The subset $\fC$ of $W$}
Let $\cJ$ be the set of all proper subsets of $S$. Note that since $W$ is an irreducible affine Weyl group, $W_{I}$ is finite for all $I\in \cJ$. We set
$$\fC:=\bigcup_{I\in \cJ} W_{I}.$$
Let $x,y\in \fC$; we write $x\sim_{\cL\cR,\fC} y$ if there exist a sequence $x=x_{0},...,x_{n}=y$ in $\fC$ and a sequence $I_{0},...,I_{n-1}$ in $\cJ$ such that
$$\text{$x_{k},x_{k+1}\in W_{I_{k}}$ and $x_{k}\sim_{\cLR}x_{k+1}$ in $W_{I_{k}}$}$$
for all $0\leq k\leq n-1$. This an equivalence relation and the equivalence classes will be called (for obvious reasons) the two-sided cells of $\fC$. We denote by $\cP_{\cL\cR,\fC}$ the associated partition of~$\fC$.\\
Let $w\in \fC$ and assume that there exists $I,J\in\cJ$ such that $x\in W_{I}\cap W_{J}=W_{I\cap J}$. Then, using {\bf P12} in $W_{I}$ and $W_{J}$, we obtain $\ba_{I}(x)=\ba_{I\cap J}(x)=\ba_{J}(x)$. Thus, we can define the following function for all $x\in \fC$
$$\ba_{\fC}(x)=\ba_{I}(x) \text{ if $x\in W_{I}$}.$$
\begin{Lem}
\label{cellC}
We have
\begin{enumerate}
\item the $\ba_{\fC}$-function is constant on the two-sided cells of $\fC$;
\item the two-sided cells of $\fC$ are connected.
\end{enumerate}
\end{Lem}
\begin{proof}
Let $b$ be a two-sided cell of $\fC$ and let $x,y\in b$.  There exist a sequence $x=x_{0},...,x_{n}=y$ in $\fC$ and a sequence $I_{0},...,I_{n-1}$ in $\cJ$ such that
$$\text{$x_{k},x_{k+1}\in W_{I_{k}}$ and $x_{k}\sim_{\cLR}x_{k+1}$ in $W_{I_{k}}$}$$
for all $0\leq k\leq n-1$. Fix such a $k$. Using {\bf P4} in $W_{I_{k}}$ we get $\ba_{\fC}(x_{k})=\ba_{\fC}(x_{k+1})$ and (1) follows. 
Let $C_{k}$ be the two-sided cell of $W_{I_{k}}$ which contains $x_{k}$ and $x_{k+1}$. By definition of $b$ we have $C_{k}\subset b$. Since $C_{k}$ is connected we get that $x_{k}$ and $x_{k+1}$ are connected in $b$ and (2) follows. 
\end{proof}
Recall the definition of $Z(w)$ in Section \ref{aprime}. For $w\in W$ we set 
$$\ba'_{\fC}(w)=\max_{u\in Z(w)} \ba_{\fC}(u).$$ 
\begin{Rem}
Note that if {\bf P12} holds in $W$, then we have $\ba'_{\fC}(w)=\ba'(w)$. 
\end{Rem}

\subsection{On the partition of $W$ into cells}
For $i\in\nN$ we denote by $B_{i}$ the subset of $W$ which consists of all the elements $w$ such that $\ba'_{\fC}(w)=i$. According to the conjectures presented in the previous section, the following should hold.
\begin{conj}
\label{dec}
The two-sided cells of $W$ are the connected components of the sets $B_{i}$ for $i\in \nN$. Furthermore the left cells lying in a given two-sided cell are the left connected components.
\end{conj}
\begin{Rem}
Assume that {\bf P4}, {\bf P9}--{\bf P11}, {\bf C1} and {\bf C2} hold in $W$. Then using {\bf P12}  and {\bf C1} we get
$$B_{i}=\{w\in W| \ba_{\fC}'(w)=i\}=\{w\in W| \ba'(w)=i\}=\{w\in W| \ba(w)=i\}.$$
It follows by {\bf P4} that $B_{i}$ is a union of two-sided cells. Now let $C$ be a connected component of $B_{i}$. By Lemma \ref{lcells}, we see that $C$ is also a union of two-sided cells. By Lemma \ref{connectedcell}, we get that $C$ is included in a two-sided cell. Hence it is a two-sided cell. The statement about the left cells follows from Lemma \ref{lcells}.
\end{Rem}
We now present an algorithm to compute the partition of $W$ into two-sided cells given in this conjecture. This algorithm uses a downward induction on the $\ba_{\fC}$-function.\\

\noindent
{\bf Step 1}: We determine the partition $\cP_{\cLR,\fC}$ of $\fC$. The $\ba_{\fC}$-function is constant on the parts of $\cP_{\cLR,\fC}$ (Lemma \ref{cellC}); we denote by $\ba_{\fC}(b)$ its value on $b\in\cP_{\cLR,\fC}$.\\

\noindent
{\bf Step 2}: Let $b_{0},\ldots,b_{m}$ be a numbering of the partition of $\cP_{\cLR,\fC}$ such that $\ba_{\fC}(b_{k})\geq \ba_{\fC}(b_{k+1})$ for all $0\leq k\leq m-1$. For $k=0,\ldots,m$ we define by induction, starting at $k=0$, the following subsets of $W$
$$\tb_{k}=\{w\in W\ |\ w=xuy,\ell(w)=\ell(x)+\ell(u)+\ell(y),\ x,y\in W, u\in b_{k}\}-\ind{\ba_{\fC}(b_{l})>\ba_{\fC}(b_{k})}{\bigcup_{l<k}} \tb_{l}.$$

\noindent
{\bf Step 3}: 
Let $x,y\in \fC$ and let $b_{i},b_{j}\in\cP_{\cLR,\fC}$ be such that $x\in b_{i}$ and $y\in b_{j}$. We write $x\sim_{\fC} y$ if $\tb_{i}\cap \tb_{j}\neq \emptyset$. We extend this relation by transitivity, we still denote it by $\sim_{\fC}$. It is an equivalence relation. We denote by $\cP_{\fC}$ the corresponding partition of $\fC$. Note that the $\ba_{\fC}$-function is constant on the parts of $\cP_{\fC}$ (see the proof of the claim below); we denote by $\ba_{\fC}(c)$ its value on $c\in\cP_{\fC}$.\\

\noindent
{\bf Step 4}: Let $c_{0},\ldots,c_{n}$ be a numbering of the partition $\cP_{\fC}$ such that $\ba_{\fC}(c_{k})\geq \ba_{\fC}(c_{k+1})$ for all $0\leq k\leq n-1$. For $k=0,\ldots,n$ we define by induction, starting at $k=0$, the following subsets of $W$
$$\tc_{k}:=\{w\in W\ |\ w=xuy,\ell(w)=\ell(x)+\ell(u)+\ell(y),\ x,y\in W, u\in c_{k}\}-\ind{\ba_{\fC}(c_{l})>\ba_{\fC}(c_{k})}{\bigcup_{l<k}} \tc_{l}.$$

\begin{Cl}
The sets $\tc_{0},\ldots,\tc_{n}$ constructed in the algorithm are the connected components of the sets $B_{i}$.
\end{Cl}

\begin{proof}
Let $\tb_{0},\ldots,\tb_{m}$ as defined in Step 2. Since $b_{k}$ is connected (Lemma \ref{cellC}) it follows that $\tb_{k}$ is connected for all $k$. Furthermore we have
$$B_{i}=\bigcup_{k, \ba_{\fC}(b_{k})=i} \tb_{k}.$$
Let $c\in \cP_{\fC}$. We have
\begin{enumerate}
\item if $x\sim_{\cL\cR,\fC} y$ then $x\sim_{\fC} y$;
\item if $x\sim_{\fC} y$ then $\ba_{\fC}(x)=\ba_{\fC}(y)$.
\end{enumerate}
Statement (1) is clear by definition. Let $x,y\in \fC$ be such that $x\sim_{\fC} y$. Let $b_{i},b_{j}\in\cP_{\cLR,\fC}$ be such that $x\in b_{i}$ and $y\in b_{j}$. Recall that 
$$\tb_{i}=\{w\in W\ |\ w=xuy,\ell(w)=\ell(x)+\ell(u)+\ell(y),\ x,y\in W, u\in b_{i}\}-\ind{\ba_{\fC}(b_{l})>\ba_{\fC}(b_{i})}{\bigcup_{l<i}} \tb_{l}.$$
Since we have $\tb_{i}\cap\tb_{j}\neq \emptyset$ it follows that $\tb_{j}$ does not appear in the union in the above formula thus we must have $\ba_{\fC}(b_{j})\leq \ba_{\fC}(b_{i})$. Similarly, we have $\ba_{\fC}(b_{i})\leq \ba_{\fC}(b_{j})$ and (2) follows.\\
Let $\tc_{0},\ldots,\tc_{n}$ as defined in Step 4. It is clear that $\tc_{k}$ is connected for all $k$ and that
$$B_{i}=\bigcup_{k, \ba_{\fC}(c_{k})=i} \tc_{k}.$$
Let $k$ be such that $\ba_{\fC}(c_{k})=i$. In order to show that $\tc_{k}$ is a connected component of $B_{i}$ it is enough to show the following statements.
\begin{enumerate}
\item Let $s\in S$ and $w\in \tc_{k}$ be such that $sw\in B_{i}$. Then $sw\in \tc_{k}$;
\item Let $s\in S$ and $w\in \tc_{k}$ be such that $ws\in B_{i}$. Then $ws\in \tc_{k}$.
\end{enumerate}
We prove (1). Let $b_{i},b_{j}\in\cP_{\cLR,\fC}$ be such that $w\in\tb_{i}$ and $sw\in\tb_{j}$. Note that $\ba_{\fC}(b_{i})=\ba_{\fC}(b_{j})$. Let
$u\in b_{i}$, $u'\in b_{j}$ and $x,y,x',y'\in W$ be such that
$$w=xuy \text{ and $\ell(w)=\ell(x)+\ell(u)+\ell(y)$},$$
$$ sw=x'u'y' \text{ and $\ell(sw)=\ell(x')+\ell(u')+\ell(y')$}.$$
Since $w\in \tc_{k}$ we have $b_{i}\subset c_{k}$ and $\tb_{i}\subset\tc_{k}$. First assume that $sw>w$. In this case we have $sw=sxuy$ and $\ell(w)=\ell(x)+\ell(u)+\ell(y)+1$. Hence, since $sw\in B_{i}$, we get $sw\in \tb_{i}\subset \tc_{k}$ and $sw\in\tc_{k}$ as desired.\\
Now assume that $sw<w$. In this case we have $w=sx'u'y'$ and $\ell(w)=\ell(x')+\ell(u')+\ell(y')+1$ and we get $w\in \tb_{j}$. It follows that $\tb_{i}\cap \tb_{j}\neq \emptyset$, $u\sim_{\fC} u'$ and $\tb_{j}\subset \tc_{k}$. Since $sw\in \tb_{j}$ we get $sw\in\tc_{k}$ as required. The proof of (2) is completely similar.
\end{proof}

\begin{Rem}
Two different weight functions can give rise to the same partition of $\fC$ into two-sided cells, however the sets $\tb_{i}$ can be different. Indeed, the definition of $\tb_{i}$ depends on the order induced by the values of the $\ba_{\fC}$-function on the two-sided cells of $\fC$.  
\end{Rem}

\subsection{Lowest two-sided cell}
A special feature of affine Weyl groups is the existence of a distinguished two-sided cell, the so-called lowest two-sided cell, which is minimal with respect to the partial order $\sim_{\cLR}$ (we refer to \cite{bremke,jeju2,Shil1,Shil2} for more information on this cell). The following proposition shows how it naturally appears in our algorithm.
\begin{Prop}
There exists a unique equivalence class $c_{0}$ with respect to $\sim_{\fC}$ with maximal $\ba$-value. Furthermore the set $\tc_{0}$ is the lowest two-sided cell of $W$.
\end{Prop}
\begin{proof}
Let $\tilde{\nu}=\max_{I\subsetneq S} L(w_{I})$ and let $\bS$ be the set which consists of all the proper subsets $I$ of $S$ such that $L(w_{I})=\tilde{\nu}$. We set 
\begin{equation*}
c_{0}=\{w_{I}| I\in \bS\}. \tag{$\ast$}
\end{equation*}
First of all, using Proposition \ref{afonc}, it is clear that $\tilde{\nu}$ is the maximal value of the $\ba_{\fC}$-function on the equivalence classes of $\fC$ with respect to $\sim_{\fC}$. It follows that $c_{0}$ has to be a union of equivalence classes. Thus in order to prove the proposition, it is enough to show that $c_{0}$ is included in an equivalence class, that is, we have $w_{I}\sim_{\fC} w_{J}$ for all $I,J\in \bS$. Let $I,J\in\bS$. There exists a unique $x\in X_{I}$ and $u\in W_{I}$ such that $w_{J}=xu$ and $\ell(w_{J})=\ell(x)+\ell(u)$. Let $v=u^{-1}w_{I}$. Note that $\ell(v)=\ell(w_{I})-\ell(u)$. Then we have $w_{J}v=xuv=xw_{I}$ and 
$$\ell(w_{J}v)=\ell(w_{J})+\ell(v)\text{ and } \ell(xw_{I})=\ell(x)+\ell(w_{I}).$$
It follows that $w_{I}\sim_{\fC} w_{J}$ as desired. Now we have
$$\tc_{0}=\{w\in W| w=xw_{I}y,\ell(w)=\ell(x)+\ell(w_{I})+\ell(y), x,y\in W, I\in\bS\}$$
and $\tc_{0}$ is the  lowest two-sided cell of $W$ (see \cite[\S 5]{bremke}).
\end{proof}


\section{Semicontinuity properties of Kazhdan-Lusztig cells}
\label{semi}
Let $(W,S)$ be an arbitrary Coxeter group. Bonnaf\'e has conjectured that the Kazhdan-Lusztig cells should satisfy some semicontinuity properties when the parameters are varying  \cite{semi}. In this section we describe this conjecture. We start by collecting some results when some of the parameters are non-positive.

\subsection{Changing signs}
Assume that $S=I\cup J$ where no element of $I$ is conjugate to an element of $J$ in $W$. Let $L$ be a weight function on $W$. Let $L'$ be the weight function defined by
$$L'(s)=
\begin{cases}
L(s) & \text{if $s\in I$}\\
-L(s) & \text{if $s\in J$}
\end{cases}
$$
Then we have (\cite[\S 2.D]{semi})
\begin{Prop}
The partition of $W$ into cells with respect to $L$ coincides with the partition of $W$ into cells with respect to $L'$. 
\end{Prop}
Thus the study of cells can be reduced to the case where the parameters are non-negative integers.

\subsection{Parameters equal to zero}
Once again assume that $S=I\cup J$ where no element of $I$ is conjugate to an element of $J$. We set 
$$\tJ:=\{wtw^{-1}| w\in W_{I},t\in J\}.$$
We denote by $\tW$ the group generated by $\tJ$. Then $(\tW,\tJ)$ is a Coxeter group \cite{Bon-Dyer,Gal}.  If $\tT\in \tJ$ we denote by $\nu(\tT)$ the unique element of $J$ such that $\tT$ is $W_{I}$-conjugate to $\nu(\tT)$. Let $L$ be a weight function on $W$ such that $L(s)=0$ for all $s\in I$. Finally, we set 
$$\tL(\tT)=L(\nu(\tT)).$$
It can be shown that if $\tT$ and $\tT'$ are conjugate in $\tW$ then $\tL(\tT)=\tL(\tT')$. It follows that $\tL$ is a weight function on $\tW$. We have (see \cite[2.E]{semi})
\begin{Prop}
The left (respectively right, respectively two-sided) cells of $W$ with respect to $L$ are of the form $W_{I}.C$ (respectively $C.W_{I}$, respectively $W_{I}.C.W_{I}$) where $C$ is a left (respectively right, respectively two-sided) cell of $\tW$ with respect to $\tL$.
\end{Prop}

\subsection{Hyperplane arrangements}
Following \cite{Bourbaki} we now introduce the notion of facets and chambers. Let $V$ be an Euclidean space of dimension $m$. Let $\fH$ be a finite set of hyperplanes in $V$. Each hyperplane $H$ of $V$ defines two half-spaces, namely the connected component of $V-H$. We say that $\la,\mu$ lie on the same side of $H$ if they lie in the same connected component of $V-H$. We define the following equivalence relation on $V$: for $\la,\mu\in V$, we write $\la\sim_{\fH} \mu$ if for all $H\in\fH$ we have either
\begin{enumerate}
\item $\la,\mu\in H$;
\item $\la,\mu$ lie on the same side of $H$.
\end{enumerate}
The equivalence classes associated to this relation are called $\fH$-facets. A $\fH$-chamber is a $\fH$-facet $\cF$ such that no point of $\cF$ lies on a hyperplane $H\in\fH$. Let $\cF$ be a $\fH$-facet, we denote by $C_{\cF}$ the set of all $\fH$-chambers $C$ such that $\cF\subset \bar{C}$ (where $\bar{C}$ denotes the closure of $C$).

\subsection{Semicontinuity conjecture}
Let $\bar{S}=\{\bar{v}_{1},...,\bar{v}_{m}\}$ be the set of conjugacy classes in $S$. Any weight function on $W$ is completely determined by its values on $\bar{S}$. Let $V$ be an Euclidean space of dimension $m$ with standard basis $v_{1},...,v_{m}$.  We identify the set of weight functions on $W$ with the set of points in $V$ with integer coordinates via
$$L\longrightarrow (L(s_{1}),...,L(s_{m}))\in V$$
where $s_{i}\in \bar{v}_{i}$ for all $i$. \\
Let $v=(n_{1},...,n_{m})\in V$ ($v\neq 0$). We denote by $H_{v}$ (or $H_{(n_{1},...,n_{m})}$) the hyperplane of $V$ orthogonal to $v$, that is the hyperplane defined by
$$\{x=(x_{1},...,x_{n})\in V|\ n_{1}x_{1}+...+n_{m}x_{m}=0\}.$$
Let $\fH$ be a hyperplane arrangement. We say that 
\begin{enumerate}
\item $\fH$ is rational if all the hyperplanes in $\fH$ have the form $H_{v}$ where $v$ has integer coordinates;
\item $\fH$ is complete if $H_{v_{i}}\in \fH$ for all $1\leq i\leq m$.
\end{enumerate}
Let $\cF$ be a $\fH$-facet. We denote by $W_{\cF}$ the standard parabolic subgroup generated by 
$$\{s\in S| L(s)=0 \text{ for all $L\in \cF$}\}.$$ 
We say that a subset $C$ of $W$ is stable by translation by $W_{I}$ ($I\subset S$) on the left  (respectively on both sides) if for all $w\in C$ we have $zw\in C$ (respectively $zwz'\in C$) for all $z,z'\in W_{I}$. \\
Finally we denote by $\cC_{\cL}(L)$ (respectively $\cC_{\cLR}(L)$) the partition of $W$ into left (respectively two-sided) cells with respect to the weight function $L$. 

We can now state Bonnaf\'e's conjecture for the partition of $W$ into cells. It is enough to state it for left and two-sided cells since the map $x\mapsto x^{-1}$ sends left cells to right cells.
\begin{conj}
\label{SC}
There exists a finite complete rational hyperplane arrangement $\fH$ of $V$ such that the following hold
\begin{enumerate}
\item If $L_{1},L_{2}$ are two weight functions belonging to the same $\fH$-facet $\cF$ then $\cC_{\cL}(L_{1})$ (respectively $\cC_{\cLR}(L_{1})$) and $\cC_{\cL}(L_{2})$ (respectively $\cC_{\cLR}(L_{2})$) coincide (we denote these partitions by $\cC_{\cL}(\cF)$ and $\cC_{\cLR}(\cF)$).
\item Let $\cF$ be an $\fH$-facet. Then the cells of $\cC_{\cL}(\cF)$ (respectively $\cC_{\cLR}(\cF)$) are the smallest subsets of $W$ which are at the same time unions of cells of $\cC_{\cL}(C)$ (resp $\cC_{\cLR}(C)$) for all $C\in C_{\cF}$ and stable by translation on the left (respectively on both sides) by $W_{\cF}$.
\end{enumerate}
\end{conj}

\begin{Rem}
If $\fH$ and $\fH'$ are two hyperplane arrangements such that Conjecture \ref{SC} holds, then it holds for $\fH\cap \fH'$. It follows that there exists (if the conjecture holds) a unique minimal finite complete rational hyperplane arrangement such that statements (1) and (2) hold. The elements of this minimal arrangement are called essential hyperplanes.
\end{Rem}


\section{Decomposition of the affine Weyl groups of rank $2$}
The main result of this section is that Conjectures \ref{dec} and \ref{SC} hold for $\tG_{2}$ and $\tB_{2}$. The methods involved for the proof of this result will be presented in the next section. We give some details in type $\tG_{2}$. In type $\tB_{2}$ there are too many distinct partitions into cells to describe them here; we refer to \cite[\S 4]{comp} for the explicit partitions. 

We use the geometric presentation of an affine Weyl group as described in \cite{bremke,Lus1,Xi}. We refer to these publications for details about this presentation. 

\subsection{Affine Weyl group of type $\tG_{2}$}
Let $(W,S)$ be the affine Weyl group of type $\tG_{2}$ with diagram and weight function given by
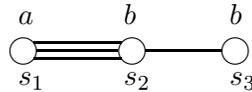
\begin{center}
\begin{picture}(150,32)
\put( 40, 10){\circle{10}}
\put( 44,  7){\line(1,0){33}}
\put( 45,  10){\line(1,0){30.5}}
\put( 44, 13){\line(1,0){33}}
\put( 81, 10){\circle{10}}
\put( 86, 10){\line(1,0){29}}
\put(120, 10){\circle{10}}
\put( 38, 20){$a$}
\put( 78, 20){$b$}
\put(118, 20){$b$}
\put( 38, -3){$s_{1}$}
\put( 78, -3){$s_{2}$}
\put(118,-3){$s_{3}$}
\end{picture}
\end{center}
where $a,b$ are positive integers. Then we have
\begin{Th}
\label{decG2}
Conjecture \ref{dec} holds for $W$.
\end{Th}

We now describe the partition of $W$ into cells. First of all, all the proper standard parabolic subgroups of $W$ are dihedral, thus looking at Section \ref{dihedral}, one can see that the decomposition of $\fC$ with respect to $\sim_{\cL\cR,\fC}$ will be constant whether $a>b$, $a=b$ or $a<b$. Furthermore we know that {\bf P1}--{\bf P15} and {\bf C2} hold for dihedral groups (see Section \ref{dihedral}). The decomposition of $\tG_{2}$ when $a=b$ has been given by Lusztig in \cite{Lus3} and one can check that Conjecture \ref{dec} holds in this case.

It is clear that the partition of $W$ into cells only depends on the ratio $a/b$. We set $r=a/b$. To simplify the notation, we denote by $W_{i,j}$ the parabolic subgroup generated by $\{s_{i},s_{j}\}$ and by $w_{i,j}$ the longest element in $W_{i,j}$.

\subsection{Case ${\bf r>1}$}
We obtain the following partition $\cP_{\cL\cR,\fC}$. 
{\small
\begin{table}[htbp] \caption{Partition $\cP_{\cL\cR,\fC}$ when $a>b$ and values of the $\ba$-function} 
\label{PlrcG}
\begin{center}
$\renewcommand{\arraystretch}{1.2}
\begin{array}{|l|c|} \hline
b_{6}=\{e\} & 0 \\
b_{5}=W_{2,3}-\{w_{2,3},e\} & b\\
b_{4}=\{w_{2,3}\}&  3b\\
b_{3}=W_{1,2}-\{e,s_{2},s_{1}s_{2}s_{1}s_{2}s_{1},w_{1,2}\}&  a\\
b_{2}=\{w_{1,3}\} & a+b \\
b_{1}=\{s_{1}s_{2}s_{1}s_{2}s_{1}\}& 3a-2b\\ 
b_{0}= \{w_{1,2}\} & 3a+3b\\ \hline
\end{array}$
\end{center}
\end{table}
}

The next step is to determine the partition $\cP_{\fC}$ and the order induced by the $\ba_{\fC}$-function on it when the parameters are varying. We describe these partitions in Table \ref{PcG} (we also put the partition in the case where $r=1$). On a given interval the values of the $\ba_{\fC}$-function are decreasing from left to right except when we write $b_{i}\overset{r'}{\leftrightarrow} b_{j}$ in which case it means that
\begin{enumerate}
\item for $a/b=r'$ we have $\ba(b_{i})=\ba(b_{j})$;
\item for all $a,b$ such that $a/b>r'$ we have $\ba(b_{i})>\ba(b_{j})$;
\item for all $a,b$ such that $a/b<r'$ we have $\ba(b_{i})<\ba(b_{j})$;
\item the sets $\tb_{i}$ and $\tb_{j}$ are disjoint.
\end{enumerate}

{\small 
\begin{table}[htbp] \caption{Partition $\cP_{\fC}$} 
\label{PcG}
\begin{center}
\renewcommand{\arraystretch}{1.3}
$
\begin{array}{|c|ccccccc|} \hline
r>2 & b_{0} & b_{1} & b_{2} & \multicolumn{2}{l}{b_{3}  \overset{3}{\leftrightarrow} b_{4}} & b_{5} & b_{6}\\ \hline

r=2&b_{0} & b_{1} & \multicolumn{2}{r}{b_{2}\cup b_{4} }& b_{3} & b_{5} & b_{6}\\ \hline

2>r>3/2 &b_{0} & \multicolumn{2}{c}{b_{1}  \overset{5/3}{ \leftrightarrow} b_{4}} & b_{2} & b_{3} & b_{5} & b_{6}\\ \hline

r=3/2 &b_{0} & b_{4} &\multicolumn{2}{r}{ b_{1} \cup b_{2}} & b_{3} & b_{5} & b_{6}\\ \hline

3/2>r>1 &b_{0} & b_{4} & b_{2} & b_{1} & b_{3} & b_{5} & b_{6}\\ \hline

r=1 &b_{0} & b_{4} & b_{2} &\multicolumn{3}{r}{ b_{1} \cup b_{3} \cup b_{5}} & b_{6}\\ \hline

 \end{array}$
\end{center}
\end{table}}

In Figure 1--6 we present the corresponding partitions of $W$; these pictures are modelled on Lusztig's picture in \cite{Lus1} for $r=1$. The left cells lying in a given two-sided cell are the left-connected components.  We also put the partition of $W$ when $r=1$. The black alcove corresponds to the identity element of $W$.\\

$\ $\\
\begin{textblock}{10}(3.6,4)

\psset{unit=0.5cm}
%
 \end{textblock}
 
$\ $\\

\newpage


\subsection{Case ${\bf r<1}$}
The partition  $\cP_{\cL\cR,\fC}$ is as follows.
{\small
\begin{table}[htbp] \caption{Partition $\cP_{\cL\cR,\fC}$ when $b>a$ and value of the $\ba$-function} 
\begin{center}
$\renewcommand{\arraystretch}{1.3}
\begin{array}{|l|c|} \hline
b_{6}=\{e\} & 0 \\
b_{5}=\{s_{1}\} & a\\
b_{4}=\fC -\{e,s_{1},s_{2}s_{1}s_{2}s_{1}s_{2},w_{1,2},w_{1,3},w_{2,3}\}&  b\\
b_{3}=  \{w_{1,3}\}& a+b\\
b_{2}=\{s_{2}s_{1}s_{2}s_{1}s_{2}\} & 3b-2a \\
b_{1}=\{w_{2,3}\}& 3b\\ 
b_{0}= \{w_{1,2}\} & 3a+3b\\ \hline
\end{array}$
\end{center}
\end{table}
}

In that case we obtain only one partition of $W$ into cells. The partition $\cP_{\fC}$ is the same as $\cP_{\cLR,\fC}$ and the order induced by the $\ba_{\fC}$-function (decreasing from left to right) is as follows
$$b_{0}\ b_{1}\ b_{2}\overset{2/3}{\leftrightarrow} b_{3} \ b_{4} \ b_{5} \ b_{6}.$$
We obtain the following decomposition of $\tG_{2}$.

\begin{textblock}{21}(0,12.5)
\begin{center}
\psset{unit=0.57cm}
%
\end{center}

\end{textblock}

\newpage

\subsection{The asymptotic case}
When $r=\infty$ (i.e. $a>0$ and $b=0$) we have
$$W= \fS_{3}\ltimes \tA_{2}$$ 
where $\fS_{3}$ is generated by $I:=\{s_{2},s_{3}\}$ and $\tA_{2}$ is generated by $s_{1},s_{2}s_{1}s_{2},s_{3}s_{2}s_{1}s_{2}s_{3}$. We know that the left cells of $W$ in this case are of the form $W_{I}.C$ where $C$ is a cell of $\tA_{2}$. We show this partition in Figure 10. The black alcoves correspond to the parabolic subgroup $W_{I}$ (i.e. to the cell $W_{I}.\{e\}$).

Now when $r=0$ (i.e. $a=0$ and $b>0$) we have
$$W= (\nZ/2\nZ)\ltimes \tA_{2}$$ 
where $\nZ/2\nZ$ is generated by $I:=\{s_{1}\}$ and $\tA_{2}$ is generated by $s_{2},s_{3},s_{1}s_{2}s_{1}$. We know that the left cells of $W$ in this case are of the form $W_{I}.C$ where $C$ is a cell of $\tA_{2}$. We show this partition in Figure 11. The black alcoves correspond to the parabolic subgroup $W_{I}$ (i.e. to the cell $W_{I}.\{e\}$).

$\ $

\begin{textblock}{10}(3.7,9.8)

\psset{unit=0.5cm}
\begin{pspicture}(-6,6)(6,-6)
\psset{linewidth=.13mm}
\SpecialCoor

\pspolygon[fillstyle=solid,fillcolor=black](0,0)(.866,1.5)(-0.866,1.5)

\pspolygon[fillstyle=solid,fillcolor=green!80!black!70!](.866,1.5)(-0.866,1.5)(-3.464,6)(-1.732,6)(0,3)(1.732,6)(3.464,6)

\pspolygon[fillstyle=solid,fillcolor=green!80!black!70!](.866,1.5)(6.062,1.5)(6.062,0)(1.732,0)(5.196,-6)(2.598,-6)(0,0)

\pspolygon[fillstyle=solid,fillcolor=green!80!black!70!](-.866,1.5)(-6.062,1.5)(-6.062,0)(-1.732,0)(-5.196,-6)(-2.598,-6)(0,0)

\pspolygon[fillstyle=solid,fillcolor=red!5!yellow!42!](0,0)(-3.464,-6)(3.464,-6)
\pspolygon[fillstyle=solid,fillcolor=red!5!yellow!42!](1.732,0)(6.062,0)(6.062,-6)(5.196,-6)
\pspolygon[fillstyle=solid,fillcolor=red!5!yellow!42!](0.866,1.5)(6.062,1.5)(6,6.062)(3.464,6)
\pspolygon[fillstyle=solid,fillcolor=red!5!yellow!42!](0,3)(1.732,6)(-1.732,6)
\pspolygon[fillstyle=solid,fillcolor=red!5!yellow!42!](-0.866,1.5)(-6.062,1.5)(-6,6.062)(-3.464,6)
\pspolygon[fillstyle=solid,fillcolor=red!5!yellow!42!](-1.732,0)(-6.062,0)(-6.062,-6)(-5.196,-6)


\SpecialCoor
\psline(0,-6)(0,6)
\multido{\n=1+1}{7}{
\psline(\n;30)(\n;330)}
\rput(1;30){\psline(0,0)(0,5.5)}
\rput(1;330){\psline(0,0)(0,-5.5)}
\rput(2;30){\psline(0,0)(0,5)}
\rput(2;330){\psline(0,0)(0,-5)}
\rput(3;30){\psline(0,0)(0,4.5)}
\rput(3;330){\psline(0,0)(0,-4.5)}
\rput(4;30){\psline(0,0)(0,4)}
\rput(4;330){\psline(0,0)(0,-4)}
\rput(5;30){\psline(0,0)(0,3.5)}
\rput(5;330){\psline(0,0)(0,-3.5)}
\rput(6;30){\psline(0,0)(0,3)}
\rput(6;330){\psline(0,0)(0,-3)}
\rput(7;30){\psline(0,0)(0,2.5)}
\rput(7;330){\psline(0,0)(0,-2.5)}
\multido{\n=1+1}{7}{
\psline(\n;150)(\n;210)}
\rput(1;150){\psline(0,0)(0,5.5)}
\rput(1;210){\psline(0,0)(0,-5.5)}
\rput(2;150){\psline(0,0)(0,5)}
\rput(2;210){\psline(0,0)(0,-5)}
\rput(3;150){\psline(0,0)(0,4.5)}
\rput(3;210){\psline(0,0)(0,-4.5)}
\rput(4;150){\psline(0,0)(0,4)}
\rput(4;210){\psline(0,0)(0,-4)}
\rput(5;150){\psline(0,0)(0,3.5)}
\rput(5;210){\psline(0,0)(0,-3.5)}
\rput(6;150){\psline(0,0)(0,3)}
\rput(6;210){\psline(0,0)(0,-3)}
\rput(7;150){\psline(0,0)(0,2.5)}
\rput(7;210){\psline(0,0)(0,-2.5)}
\multido{\n=1.5+1.5}{4}{
\psline(-6.062,\n)(6.062,\n)}
\multido{\n=0+1.5}{5}{
\psline(-6.062,-\n)(6.062,-\n)}
\psline(0;0)(7;30)

\rput(0,1){\psline(0;0)(7;30)}
\rput(0,1){\psline(0;0)(7;210)}
\rput(0,2){\psline(0;0)(7;30)}
\rput(0,2){\psline(0;0)(7;210)}
\rput(0,3){\psline(0;0)(6;30)}
\rput(0,3){\psline(0;0)(7;210)}
\rput(0,4){\psline(0;0)(4;30)}
\rput(0,4){\psline(0;0)(7;210)}
\rput(0,5){\psline(0;0)(2;30)}
\rput(0,5){\psline(0;0)(7;210)}
\rput(0,6){\psline(0;0)(7;210)}
\rput(-1.732,6){\psline(0;0)(5;210)}
\rput(-3.464,6){\psline(0;0)(3;210)}
\rput(-5.196,6){\psline(0;0)(1;210)}

\rput(0,-1){\psline(0;0)(7;30)}
\rput(0,-1){\psline(0;0)(7;210)}
\rput(0,-2){\psline(0;0)(7;30)}
\rput(0,-2){\psline(0;0)(7;210)}
\rput(0,-3){\psline(0;0)(7;30)}
\rput(0,-3){\psline(0;0)(6;210)}
\rput(0,-4){\psline(0;0)(7;30)}
\rput(0,-4){\psline(0;0)(4;210)}
\rput(0,-5){\psline(0;0)(7;30)}
\rput(0,-5){\psline(0;0)(2;210)}
\rput(0,-6){\psline(0;0)(7;30)}
\rput(1.732,-6){\psline(0;0)(5;30)}
\rput(3.464,-6){\psline(0;0)(3;30)}
\rput(5.196,-6){\psline(0;0)(1;30)}

\psline(0;0)(6.928;60)
\rput(1.732,0){\psline(0;0)(6.928;60)}
\rput(3.464,0){\psline(0;0)(5.196;60)}
\rput(5.196,0){\psline(0;0)(1.732;60)}
\rput(1.732,0){\psline(0;0)(6.928;240)}
\rput(3.464,0){\psline(0;0)(6.928;240)}
\rput(5.196,0){\psline(0;0)(6.928;240)}
\rput(6.062,-1.5){\psline(0;0)(5.196;240)}
\rput(6.062,-4.5){\psline(0;0)(1.732;240)}
\rput(-1.732,0){\psline(0;0)(6.928;60)}
\rput(-3.464,0){\psline(0;0)(6.928;60)}
\rput(-5.196,0){\psline(0;0)(6.928;60)}
\rput(-1.732,0){\psline(0;0)(6.928;240)}
\rput(-3.464,0){\psline(0;0)(5.196;240)}
\rput(-5.196,0){\psline(0;0)(1.732;240)}
\rput(-6.062,1.5){\psline(0;0)(5.196;60)}
\rput(-6.062,4.5){\psline(0;0)(1.732;60)}

\psline(0;0)(6.928;120)
\psline(0;0)(6.928;300)
\rput(-1.732,0){\psline(0;0)(6.928;120)}
\rput(-3.464,0){\psline(0;0)(5.196;120)}
\rput(-5.196,0){\psline(0;0)(1.732;120)}
\rput(-1.732,0){\psline(0;0)(6.928;300)}
\rput(-3.464,0){\psline(0;0)(6.928;300)}
\rput(-5.196,0){\psline(0;0)(6.928;300)}
\rput(-6.062,-1.5){\psline(0;0)(5.196;300)}
\rput(-6.062,-4.5){\psline(0;0)(1.732;300)}
\rput(1.732,0){\psline(0;0)(6.928;300)}
\rput(3.464,0){\psline(0;0)(5.196;300)}
\rput(5.196,0){\psline(0;0)(1.732;300)}
\rput(1.732,0){\psline(0;0)(6.928;120)}
\rput(3.464,0){\psline(0;0)(6.928;120)}
\rput(5.196,0){\psline(0;0)(6.928;120)}
\rput(6.062,1.5){\psline(0;0)(5.196;120)}
\rput(6.062,4.5){\psline(0;0)(1.732;120)}

\rput(0,1){\psline(0;0)(7;150)}
\rput(0,1){\psline(0;0)(7;330)}
\rput(0,2){\psline(0;0)(7;150)}
\rput(0,2){\psline(0;0)(7;330)}
\rput(0,3){\psline(0;0)(6;150)}
\rput(0,3){\psline(0;0)(7;330)}
\rput(0,4){\psline(0;0)(4;150)}
\rput(0,4){\psline(0;0)(7;330)}
\rput(0,5){\psline(0;0)(7;330)}
\rput(0,6){\psline(0;0)(7;330)}
\rput(-1.732,6){\psline(0;0)(2;330)}

\rput(0,-1){\psline(0;0)(7;150)}
\rput(0,-1){\psline(0;0)(7;330)}
\rput(0,-2){\psline(0;0)(7;150)}
\rput(0,-2){\psline(0;0)(7;330)}
\rput(0,-3){\psline(0;0)(7;150)}
\rput(0,-3){\psline(0;0)(6;330)}
\rput(0,-4){\psline(0;0)(7;150)}
\rput(0,-4){\psline(0;0)(4;330)}
\rput(0,-5){\psline(0;0)(7;150)}
\rput(0,-5){\psline(0;0)(2;330)}
\rput(0,-6){\psline(0;0)(7;150)}

\rput(6.062,3.5){\psline(0;0)(5;150)}
\rput(6.062,4.5){\psline(0;0)(3;150)}
\rput(6.062,5.5){\psline(0;0)(1;150)}

\rput(-6.062,-3.5){\psline(0;0)(5;330)}
\rput(-6.062,-4.5){\psline(0;0)(3;330)}
\rput(-6.062,-5.5){\psline(0;0)(1;330)}

\psline(0;0)(7;330)
\psline(0;0)(7;150)
\psline(0;0)(7;210)
\psline(0;0)(6.928;240)

\rput(0,-6.5){{\footnotesize \textsc{Figure 10.} Decomposition of $\tG_{2}$ for $r=\infty$}}

\end{pspicture}

\end{textblock}

\begin{textblock}{10}(11,9.8)

\psset{unit=0.5cm}
\begin{pspicture}(-6,6)(6,-6)
\psset{linewidth=.13mm}
\SpecialCoor

\pspolygon[fillstyle=solid,fillcolor=black](0,0)(0,1)(.866,.5)

\pspolygon[fillstyle=solid,fillcolor=blue!70!black!80!](0,1)(0,6)(.866,6)(.866,1.5)(6.062,4.5)(6.062,3.5)(.866,.5)

\pspolygon[fillstyle=solid,fillcolor=blue!70!black!80!](.866,.5)(6.062,-2.5)(6.062,-3.5)(.866,-.5)(.866,-6)(0,-6)(0,0)

\pspolygon[fillstyle=solid,fillcolor=blue!70!black!80!](0,0)(-6.062,-3.5)(-6.062,-2.5)(-.866,.5)(-6.062,3.5)(-6.062,4.5)(0,1)

\pspolygon[fillstyle=solid,fillcolor=red!5!yellow!42!](0,1)(-6.062,4.5)(-6.062,6)(0,6)
\pspolygon[fillstyle=solid,fillcolor=red!5!yellow!42!](.866,1.5)(.866,6)(6.062,6)(6.062,4.5)
\pspolygon[fillstyle=solid,fillcolor=red!5!yellow!42!](.866,.5)(6.062,3.5)(6.062,-2.5)
\pspolygon[fillstyle=solid,fillcolor=red!5!yellow!42!](-.866,.5)(-6.062,3.5)(-6.062,-2.5)
\pspolygon[fillstyle=solid,fillcolor=red!5!yellow!42!](.866,-.5)(.866,-6)(6.062,-6)(6.062,-3.5)
\pspolygon[fillstyle=solid,fillcolor=red!5!yellow!42!](0,0)(0,-6)(-6.062,-6)(-6.062,-3.5)

\SpecialCoor
\psline(0,-6)(0,6)
\multido{\n=1+1}{7}{
\psline(\n;30)(\n;330)}
\rput(1;30){\psline(0,0)(0,5.5)}
\rput(1;330){\psline(0,0)(0,-5.5)}
\rput(2;30){\psline(0,0)(0,5)}
\rput(2;330){\psline(0,0)(0,-5)}
\rput(3;30){\psline(0,0)(0,4.5)}
\rput(3;330){\psline(0,0)(0,-4.5)}
\rput(4;30){\psline(0,0)(0,4)}
\rput(4;330){\psline(0,0)(0,-4)}
\rput(5;30){\psline(0,0)(0,3.5)}
\rput(5;330){\psline(0,0)(0,-3.5)}
\rput(6;30){\psline(0,0)(0,3)}
\rput(6;330){\psline(0,0)(0,-3)}
\rput(7;30){\psline(0,0)(0,2.5)}
\rput(7;330){\psline(0,0)(0,-2.5)}
\multido{\n=1+1}{7}{
\psline(\n;150)(\n;210)}
\rput(1;150){\psline(0,0)(0,5.5)}
\rput(1;210){\psline(0,0)(0,-5.5)}
\rput(2;150){\psline(0,0)(0,5)}
\rput(2;210){\psline(0,0)(0,-5)}
\rput(3;150){\psline(0,0)(0,4.5)}
\rput(3;210){\psline(0,0)(0,-4.5)}
\rput(4;150){\psline(0,0)(0,4)}
\rput(4;210){\psline(0,0)(0,-4)}
\rput(5;150){\psline(0,0)(0,3.5)}
\rput(5;210){\psline(0,0)(0,-3.5)}
\rput(6;150){\psline(0,0)(0,3)}
\rput(6;210){\psline(0,0)(0,-3)}
\rput(7;150){\psline(0,0)(0,2.5)}
\rput(7;210){\psline(0,0)(0,-2.5)}
\multido{\n=1.5+1.5}{4}{
\psline(-6.062,\n)(6.062,\n)}
\multido{\n=0+1.5}{5}{
\psline(-6.062,-\n)(6.062,-\n)}
\psline(0;0)(7;30)

\rput(0,1){\psline(0;0)(7;30)}
\rput(0,1){\psline(0;0)(7;210)}
\rput(0,2){\psline(0;0)(7;30)}
\rput(0,2){\psline(0;0)(7;210)}
\rput(0,3){\psline(0;0)(6;30)}
\rput(0,3){\psline(0;0)(7;210)}
\rput(0,4){\psline(0;0)(4;30)}
\rput(0,4){\psline(0;0)(7;210)}
\rput(0,5){\psline(0;0)(2;30)}
\rput(0,5){\psline(0;0)(7;210)}
\rput(0,6){\psline(0;0)(7;210)}
\rput(-1.732,6){\psline(0;0)(5;210)}
\rput(-3.464,6){\psline(0;0)(3;210)}
\rput(-5.196,6){\psline(0;0)(1;210)}

\rput(0,-1){\psline(0;0)(7;30)}
\rput(0,-1){\psline(0;0)(7;210)}
\rput(0,-2){\psline(0;0)(7;30)}
\rput(0,-2){\psline(0;0)(7;210)}
\rput(0,-3){\psline(0;0)(7;30)}
\rput(0,-3){\psline(0;0)(6;210)}
\rput(0,-4){\psline(0;0)(7;30)}
\rput(0,-4){\psline(0;0)(4;210)}
\rput(0,-5){\psline(0;0)(7;30)}
\rput(0,-5){\psline(0;0)(2;210)}
\rput(0,-6){\psline(0;0)(7;30)}
\rput(1.732,-6){\psline(0;0)(5;30)}
\rput(3.464,-6){\psline(0;0)(3;30)}
\rput(5.196,-6){\psline(0;0)(1;30)}

\psline(0;0)(6.928;60)
\rput(1.732,0){\psline(0;0)(6.928;60)}
\rput(3.464,0){\psline(0;0)(5.196;60)}
\rput(5.196,0){\psline(0;0)(1.732;60)}
\rput(1.732,0){\psline(0;0)(6.928;240)}
\rput(3.464,0){\psline(0;0)(6.928;240)}
\rput(5.196,0){\psline(0;0)(6.928;240)}
\rput(6.062,-1.5){\psline(0;0)(5.196;240)}
\rput(6.062,-4.5){\psline(0;0)(1.732;240)}
\rput(-1.732,0){\psline(0;0)(6.928;60)}
\rput(-3.464,0){\psline(0;0)(6.928;60)}
\rput(-5.196,0){\psline(0;0)(6.928;60)}
\rput(-1.732,0){\psline(0;0)(6.928;240)}
\rput(-3.464,0){\psline(0;0)(5.196;240)}
\rput(-5.196,0){\psline(0;0)(1.732;240)}
\rput(-6.062,1.5){\psline(0;0)(5.196;60)}
\rput(-6.062,4.5){\psline(0;0)(1.732;60)}

\psline(0;0)(6.928;120)
\psline(0;0)(6.928;300)
\rput(-1.732,0){\psline(0;0)(6.928;120)}
\rput(-3.464,0){\psline(0;0)(5.196;120)}
\rput(-5.196,0){\psline(0;0)(1.732;120)}
\rput(-1.732,0){\psline(0;0)(6.928;300)}
\rput(-3.464,0){\psline(0;0)(6.928;300)}
\rput(-5.196,0){\psline(0;0)(6.928;300)}
\rput(-6.062,-1.5){\psline(0;0)(5.196;300)}
\rput(-6.062,-4.5){\psline(0;0)(1.732;300)}
\rput(1.732,0){\psline(0;0)(6.928;300)}
\rput(3.464,0){\psline(0;0)(5.196;300)}
\rput(5.196,0){\psline(0;0)(1.732;300)}
\rput(1.732,0){\psline(0;0)(6.928;120)}
\rput(3.464,0){\psline(0;0)(6.928;120)}
\rput(5.196,0){\psline(0;0)(6.928;120)}
\rput(6.062,1.5){\psline(0;0)(5.196;120)}
\rput(6.062,4.5){\psline(0;0)(1.732;120)}

\rput(0,1){\psline(0;0)(7;150)}
\rput(0,1){\psline(0;0)(7;330)}
\rput(0,2){\psline(0;0)(7;150)}
\rput(0,2){\psline(0;0)(7;330)}
\rput(0,3){\psline(0;0)(6;150)}
\rput(0,3){\psline(0;0)(7;330)}
\rput(0,4){\psline(0;0)(4;150)}
\rput(0,4){\psline(0;0)(7;330)}
\rput(0,5){\psline(0;0)(7;330)}
\rput(0,6){\psline(0;0)(7;330)}
\rput(-1.732,6){\psline(0;0)(2;330)}

\rput(0,-1){\psline(0;0)(7;150)}
\rput(0,-1){\psline(0;0)(7;330)}
\rput(0,-2){\psline(0;0)(7;150)}
\rput(0,-2){\psline(0;0)(7;330)}
\rput(0,-3){\psline(0;0)(7;150)}
\rput(0,-3){\psline(0;0)(6;330)}
\rput(0,-4){\psline(0;0)(7;150)}
\rput(0,-4){\psline(0;0)(4;330)}
\rput(0,-5){\psline(0;0)(7;150)}
\rput(0,-5){\psline(0;0)(2;330)}
\rput(0,-6){\psline(0;0)(7;150)}

\rput(6.062,3.5){\psline(0;0)(5;150)}
\rput(6.062,4.5){\psline(0;0)(3;150)}
\rput(6.062,5.5){\psline(0;0)(1;150)}

\rput(-6.062,-3.5){\psline(0;0)(5;330)}
\rput(-6.062,-4.5){\psline(0;0)(3;330)}
\rput(-6.062,-5.5){\psline(0;0)(1;330)}

\psline(0;0)(7;330)
\psline(0;0)(7;150)
\psline(0;0)(7;210)
\psline(0;0)(6.928;240)

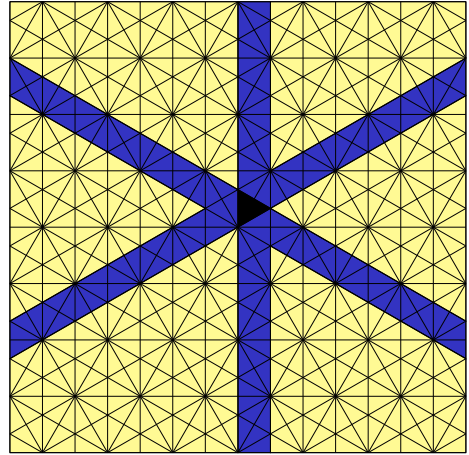
\rput(0,-6.5){{\footnotesize \textsc{Figure 11.} Decomposition of $\tG_{2}$ for $r=0$}}

\end{pspicture}

\end{textblock}

\vspace{6.5cm}

\subsection{Semicontinuity in $\tG_{2}$}
We keep the notation of Section \ref{semi}. Let $V$ be an Euclidean space of dimension 2 with standard basis $v_{1},v_{2}$ corresponding to the conjugacy classes $\{s_{1}\}$ and $\{s_{2},s_{3}\}$ in $S$, respectively.  We have
\begin{Th}
\label{G2hyp}
Conjecture \ref{SC} holds for $W$. The essential hyperplanes are 
$$\cH_{(1,0)},\cH_{(0,1)},\cH_{(1,-1)},\cH_{(1,1)},\cH_{(2,-3)},\cH_{(2,3)},\cH_{(1,-2)},\cH_{(1,2)}.$$
\end{Th}
\begin{proof}
This is straightforward once the partition of $W$ into cells is known for all choices  of parameters.
\end{proof}

\subsection{Affine Weyl group of type $\tB_{2}$}
\label{B2}
Let $(W,S)$ be the affine Weyl group of type $\tB_{2}$ with diagram and weight function given by
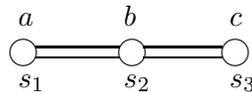
\begin{center}
\begin{picture}(150,32)
\put( 40, 10){\circle{10}}
\put( 44.5,  8){\line(1,0){31.75}}
\put( 44.5,  12){\line(1,0){31.75}}
\put( 81, 10){\circle{10}}
\put( 85.5,  8){\line(1,0){29.75}}
\put( 85.5,  12){\line(1,0){29.75}}
\put(120, 10){\circle{10}}
\put( 38, 20){$a$}
\put( 78, 20){$b$}
\put(118, 20){$c$}
\put( 38, -3){$s_{1}$}
\put( 78, -3){$s_{2}$}
\put(118,-3){$s_{3}$}
\end{picture}
\end{center}
where $a,b,c$ are positive integer. Then we have
\begin{Th}
\label{decB2}
Conjecture \ref{dec} holds for $W$.
\end{Th}

Let $V$ be an Euclidean space of dimension 3 with standard basis $v_{1},v_{2},v_{3}$ corresponding to the conjugacy classes $\{s_{1}\}$, $\{s_{2}\}$ and $\{s_{3}\}$ in $S$,  respectively. 
\begin{Th}
\label{B2hyp}
Conjecture \ref{SC} holds for $W$. Furthermore the essential hyperplanes are
$$\cH_{(1,0,0)},\cH_{(0,1,0)},\cH_{(0,0,1)},\cH_{(\eps,\eps,0)},\cH_{(0,\eps,\eps)},\cH_{(\eps,0,\eps)},\cH_{(\eps,\eps,\eps)},\cH_{(\eps,2\eps,\eps)}$$
where $\eps=\pm 1$.
\end{Th}
We refer to \cite[\S 4]{comp} for the description of the partitions $\cP_{\cLR,\fC}$ and $\cP_{\fC}$ on $\fC$ and the corresponding partitions of $W$ into cells.

\begin{Rem}
It follows from Theorem \ref{decG2} and \ref{decB2} (by a straightforward verification) that Conjecture {\bf P14} holds for $\tG_{2}$ and $\tB_{2}$.
\end{Rem}

\begin{Rem}
Let $W$ be an irreducible affine Weyl group. It seems that if one assumes that
\begin{enumerate}
\item {\bf P1}--{\bf P15} hold for all finite parabolic subgroups of $W$,
\item Conjecture \ref{dec} holds,
\item the semicontinuity conjecture holds for all finite parabolic subgroups of~$W$,
\end{enumerate}
then one could prove that the semicontinuity conjecture holds for $W$. However, the proof of this result in general looks as if it requires some deep properties of the Lusztig $\ba$-function.\\
More generally one could ask if Conjecture \ref{dec} holds for an arbitrary infinite Coxeter group $W$ and if, assuming (1)--(3), one could show that Conjecture \ref{SC} holds for $W$.
\end{Rem}


\section{Proof of Theorem \ref{decG2} and Theorem \ref{decB2}}
In this section we present the methods involved in the proof of Theorem \ref{decG2} and Theorem \ref{decB2}. We start by collecting some results of \cite{jeju3}. We then give an outline of the steps to be taken in order to prove the theorems. Finally we study an example. More details can be found in \cite[\S 3 and \S 5]{comp}.
\subsection{Prerequisites}

One of the main ingredient to prove Theorem \ref{decG2} and  \ref{decB2}  is the generalized induction of Kazhdan-Lusztig cells \cite{jeju3}. Here we will need a slightly more general version (see Remark \ref{gen}), therefore we give some details. 

Let $(W,S)$ be an arbitrary Coxeter group together with a positive weight function $L$. Let  $U\subseteq W$ be a finite subset of $W$ and let $\{X_{u}\ |\ u\in U\}$ be a collection of subsets of $W$ satisfying the following conditions
\begin{enumerate}
\item[{\bf I1}.] for all $u\in U$, we have $e\in X_{u}$,
\item[{\bf I2}.] for all $u\in U$ and $x\in X_{u}$ we have $\ell(xu)=\ell(x)+\ell(u)$,
\item[{\bf I3}.] for all $u,v\in U$ such that $u\neq v$ we have $X_{u}u\cap X_{v}v=\emptyset$,
\item[{\bf I4}.] the submodule $\cM:=\sg T_{x}C_{u}|\ u\in U,\ x\in X_{u}\sd_{\cA}\subseteq \cH$ is a left ideal.
\end{enumerate}
One can easily see that the set $\cB:=\{T_{x}C_{u}|u\in U,x\in X_{u}\}$ is a basis of $\cM$. Thus for all $y\in W$ and all $v\in U$, we can write
$$T_{y}C_{v}=\sum_{u\in U,x\in X_{u}}a_{x,u}T_{x}C_{u}\quad \text{for some $a_{x,u}\in \cA$}.$$
Let $\preceq$ be the relation on $U$ defined as follows. Let $u,v\in U$. We write $u\preceq v$ if there exist $y\in W$ and $x\in X_{u}$ such that $T_{x}C_{u}$ appears with a non-zero coefficient in the expression of $T_{y}C_{v}$ in the basis $\cB$. We still denote by $\preceq$ the pre-order induced by this relation (i.e. the transitive closure).\\
For $u,v\in U$, $x\in X_{u}$ and $y\in X_{v}$ we write $xu\sqs yv$ if $u\preceq v$ and $xu< yv$. We write $xu\sq yv$ if $xu\sqs yv$ or $x=y$ and $u=v$. Now assume that the following holds
\begin{enumerate}
\item[{\bf I5}.] for all $v\in U$, $y\in X_{v}$ we have
$$T_{y}C_{v}\equiv T_{yv}+\sum_{xu\sqs yv} a_{xu,yv}T_{x}C_{u} \mod \cH_{<0} $$
\end{enumerate}
where $a_{xu,yv}\in\cA$ and $\cH_{<0}=\oplus_{w\in W} \cA_{<0}T_{w}$.
\begin{Rem}
\label{gen}
In \cite{jeju3}, condition {\bf I5} was stated as follows
$$T_{y}C_{v}\equiv T_{yv} \mod \cH_{<0},$$
Thus Condition {\bf I5} here can be seen as a generalization of  Condition {\bf I5} in \cite{jeju3}.
\end{Rem}
Let $\fB$ be a subset of $W$. We say that $\fB$ is a left (respectively two-sided) ideal of $W$ if for all $w\in \fB$ and all $y\in W$ we have 
$$y\leq_{\cL} w\  (\text{respectively } y\leq_{\cLR} w) \Longrightarrow y\in \fB.$$
\begin{Rem}
\label{utsc}
Note that a left (respectively two-sided) ideal of $W$ is a union of left (respectively two-sided) cells. Moreover, a left ideal which is stable by taking the inverse is a two-sided ideal. 
\end{Rem}

\begin{Th}
\label{Gind}
Let $U$ be a subset of $W$ and $\{X_{u}|u\in U\}$ be a collection of subsets of $W$ satisfying conditions {\bf I1--I5}. Let $\mathcal{U}\subseteq U$ be such that for all $v\in \cU$, $u\preceq v$ implies $u\in\cU$.
Then, the set 
$$\fB:=\{yv| v\in \cU,y\in X_{v}\}$$
is a left ideal of $W$. In particular $\fB$ is a union of left cells.
\end{Th}
\begin{proof}
Arguing in the same way as in the proof of \cite[Theorem 3.2]{jeju3}, for all $v\in U$ and $y\in X_{v}$ there exists a unique family of polynomials $p^*_{xu,yv}\in\cA_{<0}$ such that
$$\tC_{yv}:=T_{y}C_{v}+\ind{xu\sqs yv}{\sum_{u\in U, x\in X_{u}}}p^{*}_{xu,yv}T_{x}C_{u}$$
is stable under the involution $\ \bar{ } \ $. It is clear that the set $\tilde{\cB}=\{\tC_{yv}| v\in U, y\in X_{v}\}$ forms a basis of $\cM$ and the above formula describes the bases change from $\cB=\{T_{y}C_{v}|v\in U,y\in X_{v}\}$ to $\tilde{\cB}$. By an easy induction on the relation $\sq$ we can invert this formula. Thus there exist some polynomials $q^{*}_{xu,yv}\in\cA_{<0}$  such that
$$T_{y}C_{v}= \tC_{yv}+\sum_{xu\sqs yv} q^{*}_{xu,yv} \tC_{xu}.$$
We want to prove by induction on $\sq$ that 
\begin{equation*}
\tC_{yv}=C_{yv}+\text{an $\cA$-linear combination of $C_{xu}$ where $xu\sqs yv$} \tag{$\ast$}.
\end{equation*}
For $u\in U$ it is clear since $\tC_{u}=C_{u}$. Now let $v\in U$ and $y\in X_{v}$. We assume that for all $u\in U$ and $x\in X_{u}$ such that $xu\sqs yv$ statement $(\ast)$ holds. By condition {\bf I5} we have
$$\tC_{yv}\equiv T_{yv}+\sum_{xu\sqs yv} a_{xu,yv}T_{x}C_{u} \mod \cH_{<0}$$
for some $a_{xu,yv}\in\cA$. From there, using the inversion formula we obtain
$$\tC_{yv}\equiv T_{yv}+\sum_{xu\sqs yv} b_{xu,yv} \tC_{xu} \mod \cH_{<0}$$
for some $b_{xu,yv}\in\cA$. Using the induction hypothesis we get
$$\tC_{yv}\equiv T_{yv}+\sum_{xu\sqs yv} b'_{xu,yv} C_{xu} \mod \cH_{<0}$$
for some $b'_{xu,yv}\in \cA$. Furthermore, since $C_{xu}\equiv T_{xu} \mod \cH_{<0}$, we may assume that $b'_{xu,yv}\in\nZ[v]$. Let $b''_{xu,yv}$ be the unique element of $\cA$ such that
$$b''_{xu,yv}=b'_{xu,yv}\mod \cA_{\leq 0}\quad\text{and}\quad \bar{b''}_{xu,yv}=b''_{xu,yv}.$$
Then we have
$$\tC_{yv}-C_{yv}-\sum_{xu\sqs yv} b''_{xu,yv} C_{xu}\equiv 0 \mod \cH_{<0}.$$
Furthermore, the left hand-side is stable under the involution $\bar{}\ $. Hence, using \cite[\S 5.(e) ]{bible}, it is equal to $0$ and ($\ast$) follows.\\
Now let $\cU$ be a subset of $U$ such that for all $v\in\cU$, $u\preceq v$ implies $u\in\cU$. We want to show that the set 
$$\fB:=\{yv| v\in \cU,y\in X_{v}\}$$
is a left ideal of $W$. One can see that 
$$\cM_{\cU}:=\sg T_{y}C_{v}\ |\ v\in \cU, y\in X_{v}\sd_{\cA}$$
is a left ideal of $\cH$. Furthermore, the set $\fB_{\cU}=\{T_{y}C_{v}| v\in \cU,y\in X_{v}\}$ is a basis of $\cM_{\cU}$. Using $(\ast)$ we see that the set $\{C_{yv}| v\in \cU, y\in X_{v}\}$ is also a basis of $\cM_{\cU}$. Let $w=yv\in\cB$ where $v\in\cU$ and $y\in  X_{v}$. Let $z\in W$ be such that $z\leq_{\cL} w$. We may assume that there exists $s\in S$ such that $C_{z}$ appears with a non-zero coefficient in the expression of $C_{s}C_{w}$ in the Kazhdan-Lusztig basis. Now since $C_{w}\in\cM_{\cU}$ and $\cM_{\cU}$ is a left ideal we have
$$C_{s}C_{w}=\sum_{u\in \cU, x\in X_{u}} \cA C_{xu}.$$
Thus there exist $u\in \cU$ and $x\in X_{u}$ such that $z=xu$ and $z\in\cB$ as desired.
\end{proof}

The following lemma is useful to find the partition of $W$ into two-sided cells once the partition into left cells is known.
\begin{Lem}
\label{itsc}
Let $T$ be a union of left cells which is stable by taking the inverse. Let $T=\cup \ T_{i}$ ($1\leq i\leq N$) be the decomposition of $T$ into left cells. Assume that for all $i,j\in\{1,...,N\}$ we have
\begin{equation*}
T_{i}^{-1}\cap T_{j}\neq \emptyset
\end{equation*}
Then $T$ is included in a two-sided cell.
\end{Lem}
The proof can be found in \cite[Lemma 5.3]{jeju3}. 

\subsection{General methods}
Let $W$ be an affine Weyl group of rank 2. Let $\cF$ be a facet of the hyperplane arrangement described in Theorem \ref{G2hyp} if $W$ is of type $\tG_{2}$ or in Theorem \ref{B2hyp} if $W$ is of type $\tB_{2}$. 
\begin{Rem}
We will have to compute some Kazhdan-Lusztig polynomials for all weight functions $L\in\cF$. This is done using \cite[Proposition 3.3]{jeju1}.
\end{Rem}
Let $\tc_{0},\ldots,\tc_{n}$ be the partition of $W$ (associated to $\cF$) obtained using the algorithm presented in Section 3. We denote by
\begin{itemize}
\item $\tc_{i}^{j}$ the left-connected components lying in $\tc_{i}$;
\item $u_{i}^{j}$ the element of minimal length lying in $\tc_{i}^{j}$;
\item $U$ the set which consists of all the $u_{i}^{j}$.
\end{itemize}
Let $u_{i}^{j}\in U$. We set
$$X_{u_{i}^{j}}:=\{w\in W| wu^{j}_{i}\in \tc^{j}_{i}\}.$$
The first step is to show that $\tc_{i}^{j}$ is included in a left cell for all $i,j$ and for all weight functions $L\in\cF$. Methods for dealing with this problem are presented in \cite[\S 6]{jeju1}. \\
The next step is to show that $c_{i}^{j}$ is a union of left cells for all weight functions $L\in\cF$. This is done as follows.
\begin{enumerate}
\item Check that  the set $U$ and the collection of subsets $X_{u}$ satisfy Conditions {\bf I1}--{\bf I4};
\item Determine the pre-order $\preceq$ on $U$. We will find that
$$\begin{array}{lc}
(\dag) &  u\preceq u_{i}^{j}\Longrightarrow u=u_{i}^{j} \text{ or }u=u_{i'}^{j'} \text{ where } i'>i.
\end{array}$$
\item Check that Condition {\bf I5} is satisfied for all $u\in U$ and all $x\in X_{u}$.
\end{enumerate}

\begin{Cl}
\label{cij}
Assume that (1)--(3) hold. Then, the sets $\tc_{i}^{j}$ are unions of lefts cells for all $i,j$.  
\end{Cl}
\begin{proof}
We proceed by induction on $i$. We can apply Theorem \ref{Gind} to the set $\{u_{0}^{j}\}$ for all $j$ and we obtain that $\tc_{0}^{j}$ is a union of left cells for all $j$. Let $i>0$ and consider the set $V=\{v\in U| v\preceq u_{i}^{j}\}$ for some $j$. Let $V'=\{v\in U| v\preceq u_{i}^{j}, v\neq u_{i}^{j}\}$. Note that if $u_{i'}^{j'}\in V'$ then $i'<i$ by ($\dag$). Applying Theorem \ref{Gind} to $V$ we obtain that
$$\{xu| u\in V, x\in X_{u}\}=\tc_{i}^{j}\cup \big(\bigcup_{u_{i'}^{j'}\in V'}  \tc_{i'}^{j'} \big)$$
is a union of left cells. Thus, since $\tc_{i'}^{j'}$ is a union of left cells for all $i'<i$ by induction, it follows that $\tc_{i}^{j}$ is a union of left cells as required. 
\end{proof}
The fact that the sets $\tc_{i}$ are the two-sided cells of $W$ will follow easily from Remark \ref{utsc} and Lemma~\ref{itsc}.
\begin{Rem}
It follows from (1)--(3) that {\bf P} holds holds for $W$ (see Remark \ref{propp}). Indeed let $x,y\in W$ be such that $x\leq_{\cL} y$ and $x\sim_{\cLR}y$. Since $x\sim_{\cLR} y$ there exists $i$ such that $x,y\in\tc_{i}$. Let $j$ be such that $y\in\tc_{i}^{j}$. Now since $x\leq_{\cL} y$, using $(\dag)$ and Theorem \ref{Gind}, we get that either $x\in\tc_{i}^{j}$ or $x\in \tc_{i'}^{j'}$ where $i'<i$. But $x\in\tc_{i}$ therefore $x\in\tc_{i}^{j}$ and $x\sim_{\cL} y$ as required.
\end{Rem}
\subsection{Example}
Let $W$ be of type $\tB_{2}$ with diagram and weight function given by
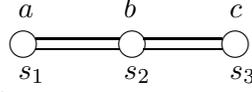
\begin{center}
\begin{picture}(150,32)
\put( 40, 10){\circle{10}}
\put( 44.5,  8){\line(1,0){31.75}}
\put( 44.5,  12){\line(1,0){31.75}}
\put( 81, 10){\circle{10}}
\put( 85.5,  8){\line(1,0){29.75}}
\put( 85.5,  12){\line(1,0){29.75}}
\put(120, 10){\circle{10}}
\put( 38, 20){$a$}
\put( 78, 20){$b$}
\put(118, 20){$c$}
\put( 38, -3){$s_{1}$}
\put( 78, -3){$s_{2}$}
\put(118,-3){$s_{3}$}
\end{picture}
\end{center}
where $(a,b,c)$ satisfy the following equations
$$b>c,\ a-2b+c>0,\ a-b-c<0.$$
In other words, if we keep the setting of Section \ref{B2}, $(a,b,c)\in V$ lies in the chamber defined by the hyperplanes
$$\cH_{(1,-2,1)},\cH_{(0,1,-1)},\cH_{(1,-1,-1)}.$$
In that case, applying Conjecture \ref{dec}, we obtain the following partition of $W$, where the left cells are formed by the alcoves lying in the same connected component after removing the thick lines. The two-sided cells are the unions of all the left cells whose names share the same subscript. The alcove corresponding to the identity is denoted by $\tc_{8}^{1}$. \\

\begin{center}
\psset{unit=1cm}
\begin{pspicture}(-4,-4)(4,4)




\psline[linewidth=.7mm](0,0)(0,1)
\psline[linewidth=.7mm](0,1)(.5,.5)
\psline[linewidth=.7mm](0.5,0.5)(0,0)

\psline[linewidth=.7mm](0,0)(0,-4)
\psline[linewidth=.7mm](0,0)(-4,-4)

\psline[linewidth=.7mm](0,2)(0,4)
\psline[linewidth=.7mm](0,2)(-2,4)

\psline[linewidth=.7mm](1,3)(1,4)
\psline[linewidth=.7mm](1,3)(2,4)

\psline[linewidth=.7mm](1,1)(4,1)
\psline[linewidth=.7mm](1,1)(4,4)

\psline[linewidth=.7mm](-1,1)(-4,1)
\psline[linewidth=.7mm](-1,1)(-4,4)

\psline[linewidth=.7mm](-2,0)(-4,0)
\psline[linewidth=.7mm](-2,0)(-4,-2)

\psline[linewidth=.7mm](1,-1)(1,-4)
\psline[linewidth=.7mm](1,-1)(4,-4)

\psline[linewidth=.7mm](2,0)(4,0)
\psline[linewidth=.7mm](2,0)(4,-2)

\psline[linewidth=.7mm](0,-1)(.5,-.5)
\psline[linewidth=.7mm](0,0)(.5,-.5)
\psline[linewidth=.7mm](.5,-.5)(1,-1)
\psline[linewidth=.7mm](.5,.5)(4,-3)
\psline[linewidth=.7mm](1,0)(2,0)

\psline[linewidth=.7mm](.5,.5)(1,1)
\psline[linewidth=.7mm](0,1)(1,1)
\psline[linewidth=.7mm](.5,1.5)(3,4)
\psline[linewidth=.7mm](.5,1.5)(0,2)
\psline[linewidth=.7mm](1,2)(1,3)
\psline[linewidth=.7mm](.5,1.5)(1,1)

\psline[linewidth=.7mm](0,1)(0,2)
\psline[linewidth=.7mm](0,1)(-3,4)
\psline[linewidth=.7mm](-.5,1.5)(-1,1)
\psline[linewidth=.7mm](-1,1)(0,1)

\psline[linewidth=.7mm](-.5,.5)(-1,1)
\psline[linewidth=.7mm](-.5,.5)(-4,-3)
\psline[linewidth=.7mm](-1,0)(-4,0)
\psline[linewidth=.7mm](-.5,.5)(0,0)

\rput(-.8,-2.5){{\small $ {\bf \tc^1_{0}}$}}
\rput(1.8,-2.5){{\small $ {\bf \tc^2_{0}}$}}
\rput(3.2,-.5){{\small $ {\bf \tc^3_{0}}$}}
\rput(3.2,2.5){{\small $ {\bf \tc^4_{0}}$}}
\rput(1.2,3.5){{\small $ {\bf \tc^5_{0}}$}}
\rput(-.8,3.5){{\small $ {\bf \tc^6_{0}}$}}
\rput(-2.8,1.5){{\small $ {\bf \tc^7_{0}}$}}
\rput(-3.2,-.5){{\small $ {\bf \tc^8_{0}}$}}

\rput(3.2,-1.5){{\small $ {\bf \tc^1_{1}}$}}
\rput(2.2,3.5){{\small $ {\bf \tc^2_{1}}$}}
\rput(-1.8,3.5){{\small $ {\bf \tc^3_{1}}$}}
\rput(-2.8,-1.5){{\small $ {\bf \tc^4_{1}}$}}

\rput(.2,-2.5){{\small $ {\bf \tc^1_{2}}$}}
\rput(2.2,.5){{\small $ {\bf \tc^2_{2}}$}}
\rput(.2,2.5){{\small $ {\bf \tc^3_{2}}$}}
\rput(-1.8,.5){{\small $ {\bf \tc^4_{2}}$}}

\rput(.2,-.5){{\small $ {\bf \tc^1_{3}}$}}

\rput(1.2,-.5){{\small $ {\bf \tc^1_{4}}$}}
\rput(1.2,1.5){{\small $ {\bf \tc^2_{4}}$}}
\rput(-1.2,1.5){{\small $ {\bf \tc^3_{4}}$}}
\rput(-1.2,-.5){{\small $ {\bf \tc^4_{4}}$}}

\rput(-.5,1.2){{\small $ {\bf \tc^1_{5}}$}}

\rput(.2,1.5){{\small $ {\bf \tc^1_{6}}$}}
\rput(-.2,.5){{\small $ {\bf \tc^2_{6}}$}}

\rput(.5,.8){{\small $ {\bf \tc^1_{7}}$}}

\rput(.2,.5){{\small $ {\bf \tc^1_{8}}$}}

\multido{\nx=-4+1}{9}{
\psline(\nx,-4)(\nx,4)}

\multido{\ny=-4+1}{9}{
\psline(-4,\ny)(4,\ny)}


\psline(-4,-3)(-3,-4)
\psline(-3,-3)(-2,-4)
\psline(-2,-3)(-1,-4)
\psline(-1,-3)(-0,-4)
\psline(0,-3)(1,-4)
\psline(1,-3)(2,-4)
\psline(2,-3)(3,-4)
\psline(3,-3)(4,-4)

\psline(-4,-2)(-3,-3)
\psline(-3,-2)(-2,-3)
\psline(-2,-2)(-1,-3)
\psline(-1,-2)(-0,-3)
\psline(0,-2)(1,-3)
\psline(1,-2)(2,-3)
\psline(2,-2)(3,-3)
\psline(3,-2)(4,-3)

\psline(-4,-1)(-3,-2)
\psline(-3,-1)(-2,-2)
\psline(-2,-1)(-1,-2)
\psline(-1,-1)(-0,-2)
\psline(0,-1)(1,-2)
\psline(1,-1)(2,-2)
\psline(2,-1)(3,-2)
\psline(3,-1)(4,-2)

\psline(-4,0)(-3,-1)
\psline(-3,0)(-2,-1)
\psline(-2,0)(-1,-1)
\psline(-1,0)(-0,-1)
\psline(0,0)(1,-1)
\psline(1,0)(2,-1)
\psline(2,0)(3,-1)
\psline(3,0)(4,-1)

\psline(-4,4)(-3,3)
\psline(-3,4)(-2,3)
\psline(-2,4)(-1,3)
\psline(-1,4)(-0,3)
\psline(0,4)(1,3)
\psline(1,4)(2,3)
\psline(2,4)(3,3)
\psline(3,4)(4,3)

\psline(-4,3)(-3,2)
\psline(-3,3)(-2,2)
\psline(-2,3)(-1,2)
\psline(-1,3)(-0,2)
\psline(0,3)(1,2)
\psline(1,3)(2,2)
\psline(2,3)(3,2)
\psline(3,3)(4,2)

\psline(-4,2)(-3,1)
\psline(-3,2)(-2,1)
\psline(-2,2)(-1,1)
\psline(-1,2)(-0,1)
\psline(0,2)(1,1)
\psline(1,2)(2,1)
\psline(2,2)(3,1)
\psline(3,2)(4,1)

\psline(-4,1)(-3,0)
\psline(-3,1)(-2,0)
\psline(-2,1)(-1,0)
\psline(-1,1)(-0,0)
\psline(0,1)(1,0)
\psline(1,1)(2,0)
\psline(2,1)(3,0)
\psline(3,1)(4,0)

\psline(-4,3)(-3,4)
\psline(-3,3)(-2,4)
\psline(-2,3)(-1,4)
\psline(-1,3)(-0,4)
\psline(0,3)(1,4)
\psline(1,3)(2,4)
\psline(2,3)(3,4)
\psline(3,3)(4,4)

\psline(-4,2)(-3,3)
\psline(-3,2)(-2,3)
\psline(-2,2)(-1,3)
\psline(-1,2)(-0,3)
\psline(0,2)(1,3)
\psline(1,2)(2,3)
\psline(2,2)(3,3)
\psline(3,2)(4,3)

\psline(-4,1)(-3,2)
\psline(-3,1)(-2,2)
\psline(-2,1)(-1,2)
\psline(-1,1)(-0,2)
\psline(0,1)(1,2)
\psline(1,1)(2,2)
\psline(2,1)(3,2)
\psline(3,1)(4,2)

\psline(-4,0)(-3,1)
\psline(-3,0)(-2,1)
\psline(-2,0)(-1,1)
\psline(-1,0)(-0,1)
\psline(0,0)(1,1)
\psline(1,0)(2,1)
\psline(2,0)(3,1)
\psline(3,0)(4,1)

\psline(-4,-1)(-3,0)
\psline(-3,-1)(-2,0)
\psline(-2,-1)(-1,0)
\psline(-1,-1)(-0,0)
\psline(0,-1)(1,0)
\psline(1,-1)(2,0)
\psline(2,-1)(3,0)
\psline(3,-1)(4,0)

\psline(-4,-2)(-3,-1)
\psline(-3,-2)(-2,-1)
\psline(-2,-2)(-1,-1)
\psline(-1,-2)(-0,-1)
\psline(0,-2)(1,-1)
\psline(1,-2)(2,-1)
\psline(2,-2)(3,-1)
\psline(3,-2)(4,-1)

\psline(-4,-3)(-3,-2)
\psline(-3,-3)(-2,-2)
\psline(-2,-3)(-1,-2)
\psline(-1,-3)(-0,-2)
\psline(0,-3)(1,-2)
\psline(1,-3)(2,-2)
\psline(2,-3)(3,-2)
\psline(3,-3)(4,-2)

\psline(-4,-4)(-3,-3)
\psline(-3,-4)(-2,-3)
\psline(-2,-4)(-1,-3)
\psline(-1,-4)(-0,-3)
\psline(0,-4)(1,-3)
\psline(1,-4)(2,-3)
\psline(2,-4)(3,-3)
\psline(3,-4)(4,-3)

\rput(0,-4.5){{\small  {\sc Figure 12.} Partition of $W$}}



\end{pspicture}

\end{center}

\vspace{.9cm}
Using the same methods as in \cite[\S 6]{jeju1} one can show that each of the pieces of this partition is included in a left cell.\\

Recall the definition of $u_{i}^{j}$, $U$ and $X_{u}$ (for $u\in U$) in the last section. We denote by $U_{i}$ the subset of $U$ which consists of all the elements $u_{i}^{j}$. Finally, for any subset $V$ of $U$ we set
$$\cM_{V}:=\sg T_{x}C_{v}|v\in V, x\in X_{v}\sd_{\cA}.$$
Our first task is to determine the pre-order $\preceq$ on $U$.
\begin{Cl}
Let $u\in U_{0}$. Then $\cM_{\{u\}}$ is a left ideal of $\cH$. 
\end{Cl}
\begin{proof}
This has been done in \cite[Lemma 5.4]{jeju2}. 
\end{proof}


\begin{Cl}
Assume that $a< 2c$. Then, the following $\cA$-submodules of $\cH$ are left ideals
$$\cM_{\{u^1_{1},u^3_{0}\}},\ \cM_{\{u^2_{1},u^5_{0}\}}, \ \cM_{\{u^3_{1},u^6_{0}\}},\ \cM_{\{u^4_{1},u^8_{0}\}}.$$
Assume that $a\geq 2c$. Then, the following $\cA$-submodules of $\cH$ are left ideals
$$\cM_{\{u^1_{1},u^3_{0}\}},\ \cM_{\{u^2_{1},u^{4}_{0}, u^5_{0}\}}, \ \cM_{\{u^3_{1},u^6_{0}\}},\ \cM_{\{u^4_{1},u_{0}^{1},u^8_{0}\}}.$$

\end{Cl}
\begin{proof}
Assume that $a<2c$. We show that $\cM:=\cM_{\{u^2_{1},u^5_{0}\}}$ is a left ideal. Using the previous claim it is enough to check that $T_{x}C_{u^2_{1}}\in\cM$ for all $x\in W$. We proceed by induction on $\ell(x)$. If $\ell(x)=0$ this is clear. Assume that $\ell(x)>0$. If $x\in X_{u^2_{1}}$ then it is clear. Thus we assume that $x\notin X_{1}$. If $x\notin X_{\{s_{2},s_{3}\}}$ then let $w\in W_{\{s_{2},s_{3}\}}$ and $x'\in X_{\{s_{2},s_{3}\}}$ such that $x=x'w$ and $\ell(x)=\ell(x')+\ell(w)$. We have
$$T_{x}C_{u^2_{1}}=v^{L(w)}T_{x'}C_{u^1_{1}}$$
which lies in $\cM$ by induction. Now assume that $x\in X_{\{s_{2},s_{3}\}}$ and that $x\notin X_{u^2_{1}}$. Then there exists $x'\in W$ such that $x=x's_{1}s_{2}s_{1}$ and $\ell(x)=\ell(x')+3$. We have
\begin{align*}
T_{s_{1}}C_{u^2_{1}}&=C_{s_{1}u^2_{1}}-v^{-a} C_{u^2_{1}}\\
T_{s_{2}s_{1}}C_{u^2_{1}}&=C_{s_{2}s_{1}u^2_{1}}-v^{-b} C_{s_{1}u^2_{1}}-v^{-a+b} C_{u^2_{1}}\\
T_{s_{1}s_{2}s_{1}}C_{u^2_{1}}&=C_{u_{0}^{5}}-v^{-a} C_{s_{2}s_{1}u^2_{1}}+v^{-2a+b} C_{u^2_{1}}
\end{align*}
From there we get that  $T_{s_{1}s_{2}s_{1}}C_{u^2_{1}}\in \cM$ and the result follows by induction and by the previous claim.\\
Now assume that $a\geq 2c$. We show that $\cM:=\cM_{\{u^2_{1},u^{4}_{0}, u^5_{0}\}}$ is a left ideal. Arguing as above it is enough to check that $T_{s_{1}s_{2}s_{1}}C_{u^2_{1}}\in \cM$. We have 
\begin{align*}
T_{s_{1}}C_{u^2_{1}}&=C_{s_{1}u^2_{1}}-v^{-a} C_{u^2_{1}}\\
T_{s_{2}s_{1}}C_{u^2_{1}}&=C_{s_{2}s_{1}u^2_{1}}-v^{-b} C_{s_{1}u^2_{1}}-v^{-a+b} C_{u^2_{1}}\\
T_{s_{1}s_{2}s_{1}}C_{u^2_{1}}&=C_{u_{0}^{5}}-v^{-a} C_{s_{2}s_{1}u^2_{1}}+v^{-2a+b} C_{u^2_{1}}+\mu C_{u_{0}^{4}}
\end{align*}
where $\mu=1$ if $a=2c$ and $\mu=v^{-a+2c}+v^{a-2c}$ if $a>c$. The result follows. We proceed similarly for the other submodules.
\end{proof}
Proceeding in the same way, we can describe the order $\preceq$ on $U$ by the following Hasse diagram. The dashed edges only hold for parameters such that $a\geq 2c$. We see that property $(\dag)$ holds.

\newpage
\begin{center}
\psset{unit=.65cm}
\begin{pspicture}(0,0)(16,16)

\pscircle(0,0){.35}
\rput(0,0){{\footnotesize $u^1_{0}$}}

\pscircle(2,0){.35}
\rput(2,0){{\footnotesize $u^2_{0}$}}

\pscircle(4,0){.35}
\rput(4,0){{\footnotesize $u^3_{0}$}}

\pscircle(6,0){.35}
\rput(6,0){{\footnotesize $u^4_{0}$}}

\pscircle(8,0){.35}
\rput(8,0){{\footnotesize $u^5_{0}$}}

\pscircle(10,0){.35}
\rput(10,0){{\footnotesize $u^6_{0}$}}

\pscircle(12,0){.35}
\rput(12,0){{\footnotesize $u^7_{0}$}}

\pscircle(14,0){.35}
\rput(14,0){{\footnotesize $u^8_{0}$}}

\pscircle(16,0){.35}
\rput(16,0){{\footnotesize $u^1_{0}$}}

\pscircle(4,2){.35}
\rput(4,2){{\footnotesize $u^1_{1}$}}

\pscircle(8,2){.35}
\rput(8,2){{\footnotesize $u^2_{1}$}}

\pscircle(10,2){.35}
\rput(10,2){{\footnotesize $u^3_{1}$}}

\pscircle(14,2){.35}
\rput(14,2){{\footnotesize $u^4_{1}$}}

\pscircle(1,4){.35}
\rput(1,4){{\footnotesize $u^1_{2}$}}

\pscircle(5,4){.35}
\rput(5,4){{\footnotesize $u^2_{2}$}}

\pscircle(9,4){.35}
\rput(9,4){{\footnotesize $u^3_{2}$}}

\pscircle(13,4){.35}
\rput(13,4){{\footnotesize $u^4_{2}$}}

\pscircle(1,6){.35}
\rput(1,6){{\footnotesize $u^{1}_{3}$}}

\pscircle(3,8){.35}
\rput(3,8){{\footnotesize $u^1_{4}$}}

\pscircle(9,8){.35}
\rput(9,8){{\footnotesize $u^2_{4}$}}

\pscircle(11,8){.35}
\rput(11,8){{\footnotesize $u^3_{4}$}}

\pscircle(13,8){.35}
\rput(13,8){{\footnotesize $u^4_{4}$}}

\pscircle(11,9.5){.35}
\rput(11,9.5){{\footnotesize $u^{1}_{5}$}}

\pscircle(9,11){.35}
\rput(9,11){{\footnotesize $u^1_{6}$}}

\pscircle(13,11){.35}
\rput(13,11){{\footnotesize $u^2_{6}$}}

\pscircle(8.5,12.5){.35}
\rput(8.5,12.5){{\footnotesize $u^{1}_{7}$}}

\pscircle(10,14){.35}
\rput(10,14){{\footnotesize $u_{8}^{1}$}}

\psline(10,13.65)(8.85,12.5)
\psline(10,13.65)(13,11.35)
\psline(10,13.65)(6,12.5)
\psline(6,12.5)(3,8.35)

\psline(4,.35)(4,1.65)
\psline(8,.35)(8,1.65)
\psline(10,.35)(10,1.65)
\psline(14,.35)(14,1.65)

\psline[linestyle=dashed](8,1.65)(6,.35)
\psline[linestyle=dashed](16,.35)(14,1.65)

\psline(1,3.65)(0,.35)
\psline(1,3.65)(2,.35)

\psline(5,3.65)(4,2.35)
\psline(5,3.65)(6,.35)

\psline(9,3.65)(8,2.35)
\psline(9,3.65)(10,2.35)
\psline(9,3.65)(6.5,2)(6,.35)

\psline(13,3.65)(12,.35)
\psline(13,3.65)(14,2.35)
\psline(13,3.65)(15.5,2)(16,.35)

\psline(1,5.65)(1,4.35)

\psline(3,7.65)(1,6.35)
\psline(3,7.65)(5,4.35)

\psline(9,7.65)(9,4.35)

\psline(11,7.65)(12,.35)
\psline(11,7.65)(10,2.35)

\psline(13,7.65)(13,4.35)

\psline(11,9.15)(11,8.35)

\psline(9,10.65)(9,8.35)

\psline(13,10.65)(11,9.85)
\psline(13,10.65)(13,8.35)

\psline(8.5,12.15)(9,11.35)
\psline(8.5,12.15)(5,4.35)

\rput(8,-1){\textsc{Figure 13.} Hasse diagram of the pre-order $\preceq$ on $U$.}
\end{pspicture}
\end{center}
$\ $\\
\begin{Cl}
Condition {\bf I5} holds for $U$ and the corresponding sets $X_{u}$.
\end{Cl}
\begin{proof}
For $u\in U_{0}$ this has been done in \cite[Lemma 5.2]{jeju2} where it is shown that
$$T_{x}C_{u}\equiv T_{xu} \mod \cH_{<0}$$
for all $u\in U_{0}$ and $x\in X_{u}$.\\

Let $u\in U_{1}$. We set $w_{1}=s_{1}s_{2}s_{3}s_{2}$. One can check that $\ell(w_{1}^n)=n\ell(w_{1})$ and that any element of $X_{u}$ can be written in the following form for some $n\in \nN$
$$w_{1}^n, s_{1}w_{1}^n, s_{2}s_{1}w_{1}^n,s_{3}s_{2}s_{1}w_{1}^n.$$
We have
$$\begin{array}{rll}
T_{s_{1}}C_{u}\equiv&T_{s_{1}u}&\mod \cH_{<0}\\
T_{s_{2}s_{1}}C_{u}\equiv&T_{s_{2}s_{1}u}&\mod \cH_{<0}\\
T_{s_{3}s_{2}s_{1}}C_{u}\equiv&T_{s_{3}s_{2}s_{1}u}&\mod \cH_{<0}\\
T_{s_{2}s_{3}s_{2}s_{1}}C_{u}\equiv&T_{s_{2}s_{3}s_{2}s_{1}u}&\mod \cH_{<0}.\\
\end{array}$$
Furthermore, if $T_{z}$ appears with a non-zero coefficient in the expression of \newline $T_{s_{2}s_{3}s_{2}s_{1}}C_{u}$ in the standard basis then we have $\ell(xz)=\ell(x)+\ell(z)$ for all $x\in X_{u}$. Hence we have 
$$T_{x}C_{u}\equiv T_{xu}  \mod \cH_{<0}$$
for all $x\in X_{u}$ and Condition {\bf I5} holds for $u\in U_{1}$.\\

Let $u\in U_{2}$ and let $w_{2}=s_{1}s_{3}s_{2}$. One can check that $\ell(w_{2}^n)=n\ell(w_{2})$ and that any element of $X_{u}$ can be written in the following form for some $n\in \nN$
$$w_{2}^n,s_{2}w_{2}^n,s_{1}s_{2}w_{2}^n,s_{3}s_{2}w_{2}^n.$$
Assume that $u\neq u^1_{2}$. Then we have
$$\begin{array}{rll}
T_{s_{2}}C_{u}\equiv&T_{s_{2}u}&\mod \cH_{<0}\\
T_{s_{1}s_{2}}C_{u}\equiv&T_{s_{1}s_{2}u}&\mod \cH_{<0}\\
T_{s_{3}s_{2}}C_{u}\equiv&T_{s_{3}s_{2}u}&\mod \cH_{<0}\\
T_{s_{1}s_{3}s_{2}}C_{u}\equiv&T_{s_{1}s_{3}s_{2}u}&\mod \cH_{<0}.\\
\end{array}$$
Furthermore, if $T_{z}$ appears with a non-zero coefficient in the expression of $T_{s_{1}s_{3}s_{2}}C_{u}$ in the standard basis then we have $\ell(xz)=\ell(x)+\ell(z)$ for all $x\in X_{u}$. Hence we have 
$$T_{x}C_{u}\equiv T_{xu}  \mod \cH_{<0}$$
for all $x\in X_{u}$ and Condition {\bf I5} holds.\\
Now let $u=u^1_{2}$. In that case we have
$$\begin{array}{rll}
T_{s_{2}}C_{u}\equiv&T_{s_{2}u}&\mod \cH_{<0}\\
T_{s_{1}s_{2}}C_{u}\equiv&T_{s_{1}s_{2}u}+v^{a-c}\ T_{s_{1}s_{2}s_{1}s_{2}}&\mod \cH_{<0}\\
T_{s_{3}s_{2}}C_{u}\equiv&T_{s_{3}s_{2}u}&\mod \cH_{<0}\\
T_{s_{1}s_{3}s_{2}}C_{u}\equiv&T_{s_{1}s_{3}s_{2}u}+v^{a-c}\ T_{s_{3}s_{1}s_{2}s_{1}s_{2}}&\mod \cH_{<0}.\\
\end{array}$$
One can show that 
$$\begin{array}{rlc}
T_{s_{1}s_{2}}C_{u}\equiv& T_{s_{1}s_{2}u}+v^{a-c}\ C_{s_{1}s_{2}s_{1}s_{2}}&\mod \cH_{<0}\\
\end{array}$$
and
$$\begin{array}{rlc}
T_{s_{1}s_{3}s_{2}}C_{u}=&T_{s_{1}s_{3}s_{2}u}+v^{a-c}\ T_{s_{3}}C_{s_{1}s_{2}s_{1}s_{2}}&\mod \cH_{<0}.\\
\end{array}$$
Furthermore if $T_{z}$ appears with a non-zero coefficient in the expression of 
$$T_{s_{1}s_{3}s_{2}}C_{u}-T_{s_{1}s_{3}s_{1}u}-v^{a-c}\ T_{s_{3}}C_{s_{1}s_{2}s_{1}s_{2}}$$
in the standard basis then we have $\ell(xz)=\ell(x)+\ell(z)$ for all $x\in X_{u}$. We obtain that
$$\begin{array}{rlc}
T_{x}T_{s_{1}s_{3}s_{2}}C_{u}=&T_{xs_{1}s_{3}s_{2}u}+v^{a-c}\ T_{xs_{3}}C_{s_{1}s_{2}s_{1}s_{2}}&\mod \cH_{<0}.\\
\end{array}$$
for all $x\in X_{u}$. Thus we see that Condition {\bf I5} holds. \\
The claim follows by applying the same methods for the other $U_{i}$'s.
\end{proof}

From there using the Hasse diagram of the pre-order $\preceq$ on $U$, Theorem \ref{Gind}, Remark \ref{utsc} and Lemma \ref{itsc}, we see that partition of $W$ presented in Figure 12 is the actual partition of $W$ into left and two-sided cells (see Claim \ref{cij}). 
Furthermore, the order on the left cells induced by $\leq_{\cL}$ is given by the Hasse diagram of the pre-order $\preceq$ on $U$ by simply replacing $u_{i}^{j}$ by the corresponding cell $\tc_{i}^{j}$.

\end{document}